\documentclass[12pt]{article}
\usepackage{amsfonts,amsmath, amssymb,latexsym}
\def\Z{{\mathbb Z}}
\def\ind{{\rm ind}}

\def\GL{{\rm GL}}
\def\ct{{\rm ct}}

\def\Gal{{\rm Gal}}

\def\O{{\mathcal O}}
\def\P{{\mathbb P}}
\def\A{{\mathbb A}}
\def\Disc{{\rm Disc}}

\def\char{{\rm char}}
\def\len{{\rm len}}
\def\Aut{{\rm Aut}}

\def\R{{\mathbb R}}
\def\F{{\mathbb F}}

\def\Q{{\mathbb Q}}

\def\Z{{\mathbb Z}}
\def\P{{\mathbb P}}
\def\F{{\mathbb F}}
\def\Q{{\mathbb Q}}
\def\C{{\mathbb C}}

\def\m1{{\textrm{\textbf{Id}}}}

\def\Res{{\textrm{Res}}}
\def\Ann{{\textrm{Ann}}}

\newtheorem{theorem}{Theorem}
\newtheorem{corollary}[theorem]{Corollary}

\newtheorem{lemma}[theorem]{Lemma}
\newtheorem{remark}[theorem]{Remark}
\newtheorem{proposition}[theorem]{Proposition}
\newenvironment{proof}{\noindent {\bf Proof:}}{$\Box$ \vspace{2 ex}}

\title{Galois groups of random integer polynomials \\ and van der Waerden's Conjecture}

\author{Manjul Bhargava}

\begin{document}

\maketitle
\begin{abstract}
Of the $(2H+1)^n$ monic integer polynomials $f(x)=x^n+a_1 x^{n-1}+\cdots+a_n$ with $\max\{|a_1|,\ldots,|a_n|\}\leq H$, how many have associated Galois group that is not the full symmetric group $S_n$?
%
%
There are clearly $\gg H^{n-1}$ such polynomials, as may be obtained by setting $a_n=0$. 
In 1936, van der Waerden conjectured that $O(H^{n-1})$ should in fact also be the correct upper bound for the count of such polynomials.
The conjecture has been known previously for degrees $n\leq 4$, due to work of van der Waerden and Chow~and Dietmann. 
The purpose of this paper is to prove van der Waerden's Conjecture for all degrees $n$.
\end{abstract}

\section{Introduction}

For any positive integer $H$, let $E_n(H)$ denote the number of monic integer polynomials $f(x)=x^n+a_1 x^{n-1} +\cdots+a_n$ of degree $n$ with $|a_i|\leq H$ for all $i$ such that the Galois group $G_f$ of $f$ is not $S_n$. Hilbert's Irreducibility Theorem implies that $E_n(H)=o(H^n)$, i.e., $100\%$ of monic polynomials of degree $n$ are irreducible and have Galois group $S_n$.  

In 1936, van der Waerden \cite{vdw} made this statement more quantitative and proved that \vspace{-.035in}
\begin{equation}\label{vdwthm}
E_n(H)= O(H^{n-\frac{1}{6(n-2)\log\log H}}).
\end{equation}
In the same paper, he also suggested the tantalizing and lasting conjecture that $E_n(H)=O(H^{n-1})$ for all $n$ (clearly $E_n(H) \gg H^{n-1}$ as can be seen by setting $a_n=0$).  

Successive improvements to van der Waerden's result (\ref{vdwthm}) were obtained by: \linebreak Knobloch \cite{Knobloch} (1956),  who proved that
 \begin{equation}
 E_n(H)= O(H^{n-\frac{1}{18n(n!)^3}});
 \end{equation}
Gallagher~\cite{Gallagher} (1973), who proved using his large sieve that 
    \begin{equation}
    E_n(H)= O(H^{n-1/2+\epsilon});
    \end{equation}
Zywina \cite{Zywina} (2010), who refined this to
    $E_n(H)= O(H^{n-1/2})$ for large $n$; Dietmann~\cite{Dietmann2} (2012), who proved using resolvent 
polynomials and the determinant method that $E_n(H)=O(H^{n-2+\sqrt2})$; and 
Anderson, Gafni, Lemke Oliver, Lowry-Duda, Shakan, \pagebreak  and Zhang~\cite{aimgroup}~(2021) who prove 
that $E_n(H)\!=\!O(H^{n-\frac23+\frac2{3n+3}+\epsilon})$.
    For~$n\leq 4$, van der Waerden's conjectured optimal upper bound of $O(H^{n-1})$ was proven by 
Chow and Dietmann~\cite{CD}.  
See also~\cite{vdw,Cohen,Serre,Dietmann,CD2,LOT,Xiao} for various other related works and advances on the problem. 

In this paper, we establish van der Waerden's Conjecture for all degrees~$n$:

\begin{theorem}\label{main1}
We have $E_n(H)=O(H^{n-1})$. 
\end{theorem}

\smallskip \noindent
{\bf Specific Galois groups.}
For any permutation group $G\subset S_n$ on $n$ letters, 
let $N_n(G,H)$ denote the number of monic integer polynomials 
$f(x)=x^n+a_1 x^{n-1} +\cdots+a_n$ with $|a_i|\leq H$ for all $i$ 
such that we have~$G_f\cong G$ as permutation groups.  
Then Theorem~\ref{main1} amounts to proving that 
$N_n(G,H)=O(H^{n-1})$ for all permutation groups $G<S_n$ on $n$ letters. 
The question of understanding the behavior of $N_n(G,H)$ for individual 
$G$ has been the subject of many works, including 
\cite{vdw, Dietmann, CD, CD2, Zywina,LOT,Chela,Widmer,Xiao}.

That $N_n(G,H)=O(H^{n-1})$ holds for {\it intransitive} groups $G$ was already 
shown by van der Waerden~\cite{vdw}, using the fact that polynomials having such Galois groups are exactly those that factor over~$\Q$. In fact, 
an exact asymptotic of the form 
\begin{equation}
\sum_{G\subset S_n\,\mathrm{intransitive}} N_n(G,H)= c_n H^{n-1}+O(H^{n-2})
\end{equation} 
for an explicit constant $c_n>0$ was obtained by Chela~\cite{Chela}. 
Meanwhile, 
Widmer~\cite{Widmer} has given excellent bounds in the case of 
{\it imprimitive} Galois groups $G$, using the fact that polynomials having 
such Galois groups are exactly those that correspond to number fields 
having~a~nontrivial subfield; specifically, his results imply that if $G\subset S_n$ is transitive but imprimitive,~then
\begin{equation}\label{imprim}
\sum_{G\subset S_n\,\mathrm{transitive\: but\: imprimitive}} N_n(G,H)= O(H^{n/2+2}). 
\end{equation} 
That this bound is essentially optimal can be seen by considering the polynomials $f(x)=x^n+a_2 x^{n-2}+a_{4}x^{n-4}+\cdots+a_n$, as in this case $\Q[x]/(f(x))$ is a field with the index 2 subfield $\Q[x]/(x^{n/2}+a_2x^{n/2-1}+a_{4}x^{n/2-2}+\cdots+a_n)$ for generic such choices of $f(x)$. 

Theorem~\ref{main1} thus reduces to showing that $N_n(G,H)=O(H^{n-1})$, when $G\neq S_n$, for 
{\it primitive} permutation groups $G$, which are the building 
blocks of all permutation groups. 
Our methods give the best known bounds for all primitive permutation 
groups $G$ when~$G<S_n$ and $n\geq 9$.
 To state the result, 
we recall the 
notion of the index $\ind(G)$ of a permutation group $G\subset S_n$. For 
$g\in G$, define the {\it index} $\ind(g)$ of $g$ by
\begin{equation*}
\ind(g):= n -\mbox{number of orbits of $g$ on $\{1,\ldots,n\}$}. 
\end{equation*}
The {\it index} $\ind(G)$ of $G$ is then defined by $\ind(G):=\min_{1\neq g\in G}\ind(g).$

The index of a permutation group $G$ was defined by Malle~\cite{Malle} to formulate 
conjectures on the number $F_n(G,X)$ of number fields of degree $n$ having 
associated Galois group $G$ and absolute discriminant at most $X$.  
Let $a(G)>0$ be any constant such that $F_n(G,X)=O( X^{a(G)})$. 
Then Malle conjectured that we may take $a(G)=1/\ind(G)+\epsilon$ for any 
arbitrarily small $\epsilon>0$, and indeed we should have 
$X^{1/\ind(G)}\ll F_n(G,X)\ll X^{1/\ind(G)+\epsilon}$ for all~$G$. 

We prove the following general bound on $N_n(G,H)$:

\begin{theorem}\label{main2}
For any permutation group $G\subset S_n$, we have 
\begin{equation*}\label{genbound}
N_n(G,H)
\!=\!O\;\!\!\Bigl(\;\!\!\min\bigl\{H^{n+1-\ind(G)}+{H}^{n-(n-1)\textstyle\frac{1-1/\ind(G)}{a(G)+1-1/{\ind(G)-u}}}\!\log^{c}\!H,
H^{(2n-2)(a(G)-u)+1}\!\log^{n-1}\!\!H\bigr\}\!\Bigr)\hspace{-.01in},
\end{equation*}
where $u=1/(n(n-1))$ if $G$ is primitive and $u=0$ otherwise, and $c=q(k,n-k)$ denotes the number of partitions of $k$ into at most $n-k$ parts where $k=\ind(G)$.
\end{theorem}

If Malle's Conjecture is true for $G$, or if we have a sufficiently small valid value for $a(G)$, then the second term in the sum above may be dropped, and we obtain simply $N_n(G,H)=O(\min\{H^{n+1-\ind(G)},H^{(2n-2)(a(G)-u)+1}\!\log^{n-1}\!\!H\})$. However, for a general almost simple group $G$ having large index, where our best known value for $a(G)$ may be relatively poor relative to its expected true value $1/\ind(G)$, the second term in the above sum may then be the dominant term.
Regardless of which of the three terms dominate, Theorem~\ref{main2} yields significantly improved upper bounds on $N_n(G,H)$ for all primitive groups~$G<S_n$ once $n\geq 9$. 

If $G=S_n$, then Theorem~\ref{main2} of course gives the correct bound since $\ind(S_n)=1$.  If $G$ is primitive and $G\neq S_n$, then $\ind(G)\geq 2$  (see~Proposition~\ref{jordan1}), and hence Theorem~\ref{main2} implies that $N_n(G,H)=O(H^{n-1})$ for all such permutation groups $G$ (including $G=A_n$), by plugging in the  known valid values $a(G)=1/2$ for $n=3$, $a(G)=1$ for $n=4,5$, and $a(G)=(n+2)/4$ for $n\geq 6$. Thus we immediately obtain Theorem~\ref{main1} from Theorem~\ref{main2}.


To obtain some further consequences, we note that 
primitive permutation groups themselves have even more elemental building blocks, which we will call {\bf elemental} primitive permutation groups, while all others are {\bf non-elemental}.  A primitive permutation group $G$ acting on a set $\Omega$ is called {\it non-elemental} if there exist positive integers $r,m,k$ (with $k\leq m/2$; and $r>1$ or $k>1$) such that $\Omega$ can be identified with $\Delta^r$ where $\Delta$ consists of $k$-element subsets of $\{1,\ldots,m\}$, and $G\subset S_m\wr S_r$ acts in the natural way as automorphisms of this structure.\footnote{The non-elemental primitive permutation groups include, in particular, what are called {\bf non-basic} primitive permutation groups in the literature, which correspond to the case $r>1$.}  We will prove that $$a(G)\ll \frac{\log^2n}{\sqrt{n}}$$ if $G$ is a non-elemental primitive permutation group acting on $n$ letters (see Theorem~\ref{basicindbound}).  Using Theorem~\ref{main2}, we thus obtain that $N_n(G,H)=O_\epsilon(H^{b\sqrt{n}\log^2\!n})$ if $G$ is a non-elemental primitive permutation group for some constant $b>0$. 

Meanwhile, if $G$ is an elemental primitive permutation group acting on $n$ letters, then by a result of Liebeck and Saxl~\cite{LS}, we have $\ind(G)\geq n/6$ and in fact $\ind(G)\geq 3n/14$ (and $\ind(G)\geq 2n/9$ if $n\neq 28$) by using a further  improvement due to Guralnick and Magaard~\cite{gm} (see~Theorem~\ref{basicindbound}). 
Using Schmidt's bound and the quantitative version of the Ekedahl sieve as in \cite{geosieve} (see Theorem~\ref{primcount}), we obtain $a(G)=(n+2)/4-1+1/\ind(G)+\epsilon$ for any $\epsilon>0$ and thus $N_n(G,H)=O_\epsilon(H^{n+1-\ind(G)}+H^{n-(n-1)(1-1/\ind(G))/((n+2)/4-u+\epsilon)})$ for a general primitive permutation group $G$. 

In particular, since, for any primitive permutation group $G\neq A_n,S_n$, we have $\ind(G)\geq 3$ when $n\geq 9$ and $\ind(G)\geq 4$ when $n\geq 16$ (see~Theorem~\ref{genindbound}), it follows that 
$N_n(G,H)=O(H^{n-2})$ when $n\geq 10$ (and $N_9(G,H)=O(H^{7.06})$), \pagebreak and $N_n(G,H)=O(H^{n-2.5})$ when $n\geq 16$. We will similarly deduce from Theorem~\ref{main2} that $N_n(G,H)=O(H^{n-3})$ when $n\geq 28$, and $N_n(G,H)=O(H^{n-3.5})$ when $n\geq 53$. 

For large $n$, the best known bound on $a(G)$ at the current time is 
$c\log^2\!n$ for some 
constant $c>0$, due to Lemke Oliver and Thorne~\cite{LOT2}, independent of $G$. Plugging $\ind(G)\geq 3n/14$ for an elemental primitive permutation group $G$ and $a(G)=c\log^2\!n$ into Theorem~\ref{main2} yields $N_n(G,H)=O(H^{n-bn/\log^2n})$ for some constant $b>0$.

We summarize these consequences in the following corollary. 

\begin{corollary}\label{main3}
Let $G\subset S_n$ be a primitive permutation group not equal to $A_n$ or $S_n$.
\begin{enumerate}
\item[{\rm (a)}] 
$N_n(G,H)=O(H^{n-2})$ for all $n\geq 10;$ \;
$N_n(G,H)=O(H^{n-2.5})$ for all $n\geq 16;$ \\[.025in]
 $N_n(G,H)=O(H^{n-3})$ for all $n\geq 28;$ \;
 $N_n(G,H)=O(H^{n-3.5})$ for all $n\geq 53.$ 

\item[{\rm (b)}] 
If $G$ is non-elemental, then 
$N_n(G,H)=O(H^{b\sqrt{n}\log^2\!n})$ for an absolute constant $b>0;$ 
\item[{\rm (c)}] 
If $G$ is elemental, then $N_n(G,H)=O(H^{n-bn/\!\log^2\!n})$ for an absolute constant~$b>0.$
\end{enumerate}
\end{corollary}

A result of Zarhin states that if an integer polynomial $f(x)$ of degree $n$ has Galois group $A_n$ or $S_n$, then the endomorphism group 
of the Jacobian 
of the hyperelliptic curve defined by $y^2=f(x)$ over $\overline{\Q}$ has endomorphism group isomorphic to $\Z$.  Corollary~\ref{main3} thus implies that an extremely small number of hyperelliptic curve Jacobians have nontrivial endomorphism group:

\begin{corollary}\label{nonan}
In the family of hyperelliptic curves $y^2=f(x):=x^n+a_1x^{n-1}+\cdots+a_n$ defined by irreducible monic degree $n$ polynomials with $a_i\in \Z$ and $|a_i|\leq H$, at most $O(H^{n-bn/\!\log^2\!n})$ of their Jacobians have endomorphism group larger than~$\Z$.
\end{corollary}

We note that Theorem~\ref{main2} yields the best known bounds on $N_n(G,H)$ for all primitive~groups~$G$ of degree at least 9 and also for various families of imprimitive and intransitive groups $G$.  We give three examples to illustrate,
the first two of which improve upon the corresponding results in \cite{LOT}. 

We consider first the case of cyclic groups of prime order. Theorem~\ref{main2} implies:

\begin{corollary}\label{cpbound}
Let $p$ be a prime, and let $C_p\subset S_p$ denote the cyclic group of 
order $p$.  Then $N_p(C_p,H)=O(H^2)$.
\end{corollary}
Corollary~\ref{cpbound} follows from Theorem~\ref{main2} by setting
$a(G)=1/(p-1)$ (which is known to be valid due to work of Maki~\cite{Maki}; see also Wright~\cite{Wright})
and setting $\ind(G)=p-1$.  This improves upon the bound 
$N_p(C_p,H)=O(H^{3-2/p}\log^{p-1} B)$ as obtained by Lemke Oliver and Thorne in~\cite[Theorem~1.1]{LOT}. 

We next consider the case of {\bf regular} permutation groups. A permutation group $G$ acting on $n$ letters is called {\it regular} if $|G|=n$. That is, an irreducible polynomial $f\in\Q[x]$ of degree $n$ has a Galois group $G\subset S_n$ that is a regular permutation group if and only if the field extension $\Q[x]/(f(x))$ of $\Q$ is Galois. Note that such regular permutation groups are transitive but not necessarily primitive. Theorem~\ref{main2} implies:
\begin{corollary}\label{galbound}
Let $G$ be a regular permutation group with $|G|=n$.  Then 
$N_n(G,H)=O(H^{3n/11\,+\,1.164}).$
\end{corollary} 
Corollary~\ref{galbound} follows from Theorem~\ref{main2} by setting
$a(G)=3/8+\epsilon$ for \pagebreak an arbitrarily small $\epsilon>0$ (which was proven to be a valid choice of $a(G)$ by Ellenberg and Venkatesh~\cite{EV}) and setting $\ind(G)
=(p-1)n/p$  where $p$ is the smallest prime dividing $n$. This improves upon the bound $N_n(G,H)=O_\epsilon(n^{3n/4+1/4+\epsilon})$ as obtained
in~\cite[Theorem~1.2]{LOT}. 

To illustrate how the methods behind Theorem~\ref{main2} can yield improved results when applied to specific permutation groups that are intransitive, we use the following:

\begin{theorem}\label{intransitivethm}
Let $G$ be an intransitive permutation group on $n$ letters, and write the action of $G$ as a subdirect product $G\subset G_1\times\cdots\times G_k$ where $G_i\subset G$ is a transitive permutation group on $n_i$ letters, so that $n=n_1+n_2+\cdots+n_k$. Suppose $N_{n_i}(G_i,H)=O(H^{\alpha_i}\log^{\beta_i}\!H)$. Then 
$$N_n(G_1\times\cdots\times G_k,H)=O(H^{\max\{\alpha_1,\ldots,\alpha_k\}}\;\!\!\log^{s-1}\!H),$$
where $s=\sum_{\alpha_i=\max\{\alpha_1,\ldots,\alpha_k\}} (1+\beta_i)$.
\end{theorem}
We expect that Theorem~\ref{intransitivethm} is essentially optimal, giving the optimal exponent for intransitive groups $G$ whenever the optimal exponents for its subdirect factors $G_i$ are known. Thus $N_n(G,H)$ can be studied via the quantities $N_{n_i}(G_i,H)$, on which one can then apply~(\ref{imprim}) and Theorem~\ref{main2} (and the methods behind them) to obtain a good estimate on $N_n(G,H)$. This, in most cases, yields essentially the best known bounds for intransitive groups $G$ as~well. 

As an example, consider the intransitive group $G\times M_{11} \subset S_{k+11}$ (i.e., the product of a permutation group $G$ acting on $k$ letters and the Mathieu group $M_{11}$ in its natural permutation action on 11 letters), and suppose $k\leq 9$. The method of van der Waerden gives $N_{k+11}(G\times G',H)=O(H^{11})$ for any permutation groups $G\subset S_k$ and $G'\subset S_{11}$; for general such $G,G'$, this is in fact best possible, achieved when $G=S_k$ and $G'=S_{11}$. However, for our particular choice of $G'=M_{11}\subset S_{11}$, Theorem~\ref{intransitivethm} in conjunction with Theorem~\ref{main2} yields:
\begin{corollary}\label{intransitivex}
Let $G\subset S_k$ with $k\leq 9$.  Then $N_{k+11}(G\times M_{11},H)=O(H^{8.686})$.
\end{corollary}
In particular, Corollary~\ref{intransitivex} implies that $N_{11}(M_{11},H)=O(H^{8.686})$. Corollary~\ref{intransitivex} follows from Theorem~\ref{intransitivethm} and Theorem~\ref{main2} by noting that $\ind(M_{11})=4$ and that we may take $a(M_{11})=13/4-1+1/4=2.5$ as per Theorem~\ref{primcount}.



\bigskip
\noindent
{\bf{Methods and sketch of proof.}} For an irreducible integer monic polynomial $f(x)$ of degree~$n$, let $K_f$ denote the number field $\Q[x]/(f(x))$ of degree $n$, let $C$ denote the product of the primes that are ramified in $K_f$, and let $D$ denote the absolute discriminant of $K_f$.  To prove Theorem~\ref{main1}, we divide the set $P(H)$ of all irreducible integer monic polynomials $f$ of degree~$n$, having coefficients bounded by $H$ in absolute value and primitive Galois group not equal to~$S_n$, into three subsets $P_1(H)$, $P_2(H)$, and $P_3(H)$; we then bound the number of $f$ lying in the subset $P_i(H)$ by $O(H^{n-1})$, via a different method for each $i=1,2,3$. 

\medskip
\noindent {\bf Case I.}
The first subset $P_1(H)$ consists of those polynomials in $P(H)$ such that $C\leq H^{1+\delta}$ but $D>H^{2+2\delta}$ (for some positive $\delta>0$).  
We handle this case by estimating the number of monic integer polynomials in $P(H)$ that satisfy those congruence conditions modulo~$C$ that imply that the absolute discriminant of $f$ is at least $D$. For $C<H$, an estimate of $O(c^{\omega(D)}H^{n}/D)$ is immediate for a suitable constant $c>0$  (where $c^{\omega(D)}/D$ is the density of the congruence conditions), since the congruence conditions are defined modulo an integer~$C$ that is smaller than the side of the box in which we are counting. 
We then use Fourier analysis to extend the validity of this estimate even when $C<H^{1+\delta}$ for a positive $\delta$. Because we have reduced to the case of primitive Galois groups, we know that the set of possible integers $D$ that can arise are all squarefull; the sum of the estimate $O(c^{\omega(D)}H^{n}/D)$ over all $C<H^{1+\delta}$ and all squarefull $D\geq H^{2+2\delta}$ is $O_\epsilon(H^{n-1-\delta+\epsilon})$, thus yielding the desired estimate in this case.

\medskip
\noindent {\bf Case II.}
The second subset $P_2(H)$ consists of those polynomials in $P(H)$ where $D<H^{2+2\delta}$.  In this case, we use the fact that the number of possible number fields $K_f$ (up to isomorphism) is bounded because the discriminant of $K_f$ is bounded.  We use the well-known bound of Schmidt~\cite{Schmidt} to conclude that the number of possibilities for $K_f$ is at most $O(H^{(2+2\delta)(n+2)/4})$. A result of Lemke Oliver and Thorne~\cite{LOT} then implies that the number of possibilities for $f$ having a given value of $K_f$ is at most $O(H\log^{n-1}\!H/H^{1/(n(n-1))})$. Multiplying this upper bound for the number of possible $K_f$'s with the number of possible $f$'s given $K_f$ then immediately yields the upper bound $O(H^{n-1})$ for the cardinality of $P_2(H)$ when $n\geq 6$. For $n=4,5$, we have the improved bound of $O(H^{2+2\delta})$ on the number of possible $K_f$'s, by the main theorems of \cite{dodqf,dodpf}, yielding the desired result in these cases as well. 

\medskip
\noindent {\bf Case III.}
The third and final subset $P_3(H)$ consists of those polynomials in $P(H)$ for which $C>H^{1+\delta}$ (implying that $D>H^{2+2\delta}$).  In this case, we observe that a polynomial $f\in P_3(H)$ has the property that {\it $\Disc(f)$ is a multiple of $C^2$ for mod~$C$ reasons}; i.e., if $f\equiv g$ (mod~$C$), then $\Disc(g)$ is also a multiple of $C^2$. If we write $f(x)=x^n+a_1x^{n-1}+\cdots+a_n$, then we show that: if $\Disc(f)$ is a multiple of $C^2$ for mod~$C$ reasons, then the double discriminant $\Disc_{a_n}(\Disc_x(f(x)))$ (which is a polynomial only in $a_1,\ldots,a_{n-1}$) must be a multiple of $C$. Thus for generic choices of $a_1,\ldots,a_{n-1}$ in $[-H,H]^{n-1}$, the value of $C\mid \Disc_{a_n}(\Disc_x(f(x)))$ is determined by $a_1,\ldots,a_{n-1}$ up to $O_\epsilon(H^{\epsilon})$ possibilities; then $a_n$ is also generically determined by $a_1,\ldots,a_{n-1},C$ up to $O_\epsilon(H^{\epsilon})$ possibilities, being a root of the polynomial equation $\Disc(f)\equiv 0$ (mod $C$) (where $C>H$).  This enables us to conclude that the total number of possibilities for $a_1,\ldots,a_{n-1},a_n$, and thus $f$, in this case is at most $O_\epsilon(H^{n-1}\cdot H^{\epsilon}\cdot H^{\epsilon})=O_\epsilon(H^{n-1+\epsilon})$.

To remove the $\epsilon$, and thereby obtain the desired estimate $O(H^{n-1})$ on the set $P_3(H)$ and thus $P(H)$, we choose a suitable factor $C'$ of $C$ that is between $H^{1+\delta/2}$ and $H^{1+\delta}$ in size, in which case we can handle it by a method analogous to Case I.  Otherwise, we may choose a factor $C'>H$ of $C$ all of whose prime divisors are greater than $H^{\delta/2}$, in which case we can handle it by a method analogous to Case III above, but with no $\epsilon$ occurring because $C'$ will have at most a bounded number of prime factors!

\medskip
Theorem~\ref{main2} is proven in an analogous manner.  We redefine $P(H):=P(G,H)$ to be the set of monic integer polynomials of degree $n$ having coefficients bounded by $H$ in absolute value and Galois group $G$. We then divide $P(H)$ into three subsets $P_i(H):=P_i(G,H)$ for $i=1,2,3$.  We suitably redraw the boundaries between the sets $P_1(H)$, $\,P_2(H)$, and $P_3(H)$---via suitable inequalities on $C$ and $D$ for each $P_i(H)$---in a way that optimizes the resulting estimate $N_n(G,H)$.  The set $P_1(H)$ corresponds to a slightly modified Case I where $C\leq H^{1+\delta}$ but~$D>Y$ for some positive $Y$ to be optimized later. 
The set $P_2(H)$ corresponds similarly to a modified Case~II where $D\leq Y$, while $P_3(H)$ again corresponds to Case III where 
$C> H^{1+\delta}$. We apply the same arguments as in the proof of Theorem~\ref{main1} to each case, and then choose the value of $Y$ that gives the best bound on $N_n(G,H)$ obtained as the sum of the three cases.  This yields the bound stated in Theorem~\ref{main2}.

To prove Theorem~\ref{intransitivethm}, we use the fact that the height, i.e., the maximum of the absolute values of the coefficients, of an integer polynomial, up to positive bounded factors, is equal to the Mahler measure of the polynomial---and the Mahler measure is multiplicative.  We use this fact to prove that most polynomials $f$ counted by $N_n(G,H)$, for an intransitive group~$G$, occur (up to factors of $\log H$) when one factor of $f$ has height around $H$ while all other factors have small coefficients.  From this we deduce the desired estimate as stated in Theorem~\ref{intransitivethm}. 

To obtain Corollary~\ref{main3}, we appeal to the theory of primitive groups $G$, and study the indices~$\ind(G)$.  In particular, we develop new bounds on the number of number fields of bounded discriminant having primitive Galois group $G$, using the method of Schmidt~\cite{Schmidt} together with the quantitative geometric sieve methods of \cite{geosieve}. We obtain a particularly substantial improvement in the case where $G$ is non-elemental by relating $\ind(G)$ to $\ind(G')$, where $G'$ is an imprimitive permutation representation with the same underlying group as~$G$. This then allows us to concentrate on the elemental case, where far stronger lower bounds on $\ind(G)$ hold true and which may be proven using the work of Liebeck and Saxl~\cite{LS} and Guralnick and Magaard~\cite{LM}.  These bounds on $a(G)$ and $\ind(G)$ in the cases of non-elemental and elemental groups, respectively, together with the bounds of Schmidt~\cite{Schmidt} (and its modified form as stated in Theorem~\ref{primcount}) and Lemke Oliver and Thorne~\cite{LOT2} on the number of $G$-fields of bounded discriminant in the elemental case, then yield the stated bounds in Corollary~\ref{main3}, using Theorem~\ref{main2}. Meanwhile, as already noted, the best known bounds on $a(G)$ and the known values of $\ind(G)$ for specific families of groups $G$ immediately yield Corollaries~\ref{cpbound}, \ref{galbound}, and~\ref{intransitivex} from Theorems~\ref{main2} and~\ref{intransitivethm}. 

 We note that---as with methods used by other authors previously---our methods may also be applied in an essentially identical manner to study the family of 
non-monic polynomials $f(x)=a_0x^n+a_1 x^{n-1}+\cdots+a_n$ of degree $n$ 
satisfying $\max\{|a_0|,|a_1|,\ldots,|a_n|\}\leq H$. For example, if $E_n^*(H)$ denotes the number of such polynomials having Galois group not equal to $S_n$, then we may show similarly that $E_n^*(H)=O(H^n)$ (improving on the previously best known estimate  $O_\epsilon(H^{n+\frac13+\frac{8}{9n+21}+\epsilon})$ due to~\cite{aimgroup}), thus proving van der Waerden's Conjecture also in this non-monic context (that the number of reducible polynomials gives the correct order of magnitude for the count of all polynomials having Galois group not equal to $S_n$).  The analogues of Theorem~\ref{main2} may be similarly developed for this family as well as for families of polynomials where some of the coefficients $a_i$ are fixed. The exact form of these analogues would depend (only slightly) on the strength of the analogues of the results of Section~4 and \cite{LOT2} that one could prove for these families. 

 \bigskip
 \noindent
 {\bf Organization.}  
 In Sections~2--4, we collect the results that we will need regarding primitive permutation groups, counting number fields with given Galois group, and equidistribution via Fourier transforms, respectively.  In Section 5, we then carry out the proof of Theorem~\ref{main1}.  In Section 6, we describe how variations of the techniques that prove Theorem~\ref{main1} can be used to also deduce Theorem~\ref{main2}.  Theorem~\ref{main2} also immediately implies Corollaries~\ref{cpbound} and \ref{galbound} using the known values of $\ind(G)$ and $a(G)$ in these cases.
In Section 7, we prove Corollary~\ref{main3} using Theorem~\ref{main2} and the group-theoretic results of Sections~2 and 3. Corollary~\ref{main3} then immediately implies Corollary~\ref{nonan} as well.  
Finally, in Section~8, we prove Theorem~\ref{intransitivethm} in order to handle intransitive groups; Theorem~\ref{intransitivethm} (together with Theorem~\ref{main2}) then immediately implies Corollary~\ref{intransitivex}.

\section{Preliminaries on primitive permutation groups}

Our first proposition, due to Jordan (1873), states that a primitive permutation group $G\neq S_n$ cannot contain a transposition (see, e.g., {\cite[Theorem~3.3A]{dm}}).  
\begin{proposition}[Jordan]\label{jordan1}
If $G\subset S_n$ is a primitive permutation group on $n$ letters and contains a transposition, then $G=S_n$.
\end{proposition}

\begin{proof}
Suppose $G\subset S_n$ is a primitive permutation group acting on $[n]:=\{1,\ldots,n\}$, and define an equivalence relation $\sim$ on $[n]$ by asserting that $i\sim j$ if the transposition $(ij)\in G$. This equivalence relation is preserved by $G$, but since $G$ is primitive, it must be a trivial equivalence relation where all elements of $[n]$ are either mutually equivalent or mutually inequivalent. It follows that either $G$ has no transpositions or $G=S_n$.
\end{proof}

\noindent
Similarly, a primitive permutation group $G\subset S_n$ that contains a 3-cycle or a double transposition must be either $A_n$ or $S_n$, provided that $n\geq 9$ (see, e.g.,~{\cite[Theorem~3.3A \& Example~3.3.1]{dm}}):
\begin{proposition}[Jordan]\label{jordan2}
If $G\subset S_n$ is a primitive permutation group containing a $3$-cycle or a double transposition, and $n\geq 9$, then $G=A_n$ or $S_n$. 
\end{proposition}

\noindent 
Propositions~\ref{jordan1} and \ref{jordan2} imply that if a primitive permutation group $G$ on $n$ letters is not $S_n$, then $\ind(G)\geq 2$; and if $G$ is not $S_n$ or $A_n$ and $n\geq 9$, then $\ind(G)\geq 3$.  

A key relation that we will make critical use of, between indices of groups and discriminants of number fields, is contained in 
the following proposition.

\begin{proposition}\label{valindcon}
Let $f$ be an integer polynomial of degree $n$, and let $K_f:=\Q[x]/(f(x))$. If~$\ind(\Gal(f))=k$, then 
the discriminant $\Disc(K_f)$ of $K_f$ is $k$-powerful.\footnote{Recall that an integer $N$ is called {\it $k$-powerful} if for every prime $p\mid N$, we have $p^k \mid N$.}
\end{proposition}

\begin{proof}
The Galois group $G=\Gal(f)$ acts on the $n$ embeddings of $K_f$ into its Galois closure. 
Suppose $p$ factors in $K_f$ as $\prod P_i^{e_i}$, where $P_i$ has residue field degree~$f_i$. 
If $p$ is tamely ramified, then 
any generator $g\in G\subset S_n$ of an inertia group $I_p\subset G$ at $p$ is the product of disjoint cycles consisting of $f_1$ cycles of length $e_1$, $f_2$ cycles of length $e_2$, etc.; thus the discriminant valuation $v_p(\Disc(K_f))=\sum (e_i-1) f_i$ is equal to $\ind(g)\geq \ind(G)=k$.  If $p$ is wildly ramified, then $v_p(D)>\sum (e_i-1) f_i \geq \ind(g)\geq k$ for any $g\in I_p$.  In all cases, we see that $\Disc(K_f)$ is~$k$-powerful. 
\end{proof}

\noindent
Propositions~\ref{jordan1}, \ref{jordan2}, and~\ref{valindcon} together imply that if an irreducible polynomial $f$ of degree $n$ has primitive Galois group $G$ that is not $S_n$, then the discriminant $\Disc(K_f)$ is {\it squarefull}, i.e., 2-powerful; and if $G\neq  A_n,S_n$ and $n\geq 9$, then $\Disc(K_f)$ is {\it cubefull} or 3-powerful.  

Thus, proving bounds on indices of Galois groups can be an important ingredient in understanding discriminants of number fields. We have the following extension of Propositions~\ref{jordan1} and~\ref{jordan2}, which states that if $G\neq A_n,S_n$ is a primitive permutation group on $n$ letters, then in fact $\ind(G)\geq \lfloor\sqrt{n}\rfloor$. 
\begin{theorem}
\label{genindbound}
If $G\subset S_n$ is a primitive permutation group not equal to $A_n$ or $S_n$, then $\ind(G)\geq \lfloor\sqrt{n}\rfloor$.
\end{theorem}
\begin{proof} 
We deduce this elegant bound from a result of Liebeck and Saxl~\cite[Cor.~3]{LS} (see also Babai~\cite[6.14]{Babai}), who prove that the smallest number of points moved by any non-identity element in a primitive permutation group $G\subset S_n$ is strictly greater than~$2(\sqrt{n}-1)$. 

Now if $g$ is an element that moves strictly greater than $2(\sqrt{n}-1)$ elements, then $\ind(g)$ is strictly greater than~$\sqrt{n}-1$ (the latter bound being optimal when $g$ is the product of disjoint transpositions). But $\ind(g)>\sqrt{n}-1$ is equivalent to $\ind(g)\geq \lfloor \sqrt{n}\rfloor$.
\end{proof}

\noindent The bound in Theorem~\ref{genindbound} is optimal for all square values of $n\geq 25$, as can be seen by taking $G\subset S_{m^2}$ to be $S_m\wr S_2$ acting by permuting the rows and columns of an $m\times m$ matrix of~$m^2$ letters for any $m\geq 5$; then $G$ is primitive and $\ind(G)=m$.  
 
 If we restrict to elemental primitive permutation groups $G$, then a significantly stronger estimate holds:

\begin{theorem}\label{basicindbound}
If $G\subset S_n$ is an elemental primitive permutation group not equal to $A_n$ or $S_n$, then $\ind(G)\geq 3n/14$.
\end{theorem}

\begin{proof}
By a result of Guralnick and Magaard~\cite{gm} (building on a related result of Liebeck and Saxl~{\cite[Theorem~2]{LS}}), the smallest number of points moved by any non-identity element in an {\it elemental} primitive permutation group $G\subset S_n$ that is not $A_n$ or $S_n$ is at least $3n/7$, with equality when $G=O_7(2)$ acts primitively on $n=28$ elements. 

Now if $g$ is an element that moves at least $3n/7$ elements, then $\ind(g)\geq 3n/14$, with equality occurring in the case that $g$ is the product of $3n/14$ disjoint transpositions; this latter case does occur when $G=O_7(2)$, $n=28$, and $g\in G$ is a transvection. This yields the desired result. 
\end{proof}

\vspace{-.15in}
\begin{remark}\label{basicindboundgen}{\em 
The proof of Theorem~\ref{basicindbound} shows that the bound $\ind(G)\geq 3n/14$ is sharp, with equality occurring when $G=O_7(2)$ and $n=28$.  If $n\neq 28$, then we can similarly deduce from \cite{gm} that the number of points moved by any non-identity element $g$ is at least $4n/9$, achieved when $G=O_8^-(2)$, $n=27$, and $g$ is a transvection.  Hence $\ind(G)\geq 2n/9$ if $n\neq 28$.
Similarly, we may deduce that $\ind(G)\geq 7n/30$ if $n>28$ (with equality occurring when $G=O_8^+(2)$ and $n=120$), and $\ind(G)\geq n/4$ if $n<27$.}
\end{remark}

\vspace{-.15in}
\begin{remark}\label{cfsg}{\em 
The proofs of Theorems~\ref{genindbound} and~\ref{basicindbound} and Remark~\ref{basicindboundgen} rely on the classification of finite simple groups (CFSG). 
However, that these results hold for $n\leq 35$ can be deduced, without CFSG, using the  arguments above together with the classification of all primitive groups $G\subset S_n$ that have a non-identity element $g\in G$ moving $k$ letters, where $k\leq 14$; this~classification was sketched by Jordan and completed in works of Manning~\cite{Manning2p,Manning3p,Manning10,Manning68,Manning12,Manning14} (see also~\cite{Williamson} for an elegant treatment in the case that $k$ is prime).
 
The aforementioned general result of Babai~\cite[6.14]{Babai} also does not use CFSG; it states that the smallest number of points moved by a non-identity element in a primitive permutation group $G\subset S_n$ that is not $A_n$ or $S_n$ is at least $\frac12(\sqrt{n}-1)$, implying that $\ind(G)\geq \frac14(\sqrt{n}-1)$. This result is nearly as strong as Theorem~\ref{genindbound} and gives the correct order of magnitude. 

We thus check that none of the displayed Theorems and Corollaries in the introduction depend on CFSG except for the specific threshold 53 in the last assertion of Corollary~\ref{main3}(a), which may be safely replaced by 1090, and the bound $O(H^{n-bn/\!\log^2\!n})$ in Corollary~\ref{main3}(c), which may be replaced by $O(H^{n-b\sqrt{n}})$, by using only the above results of Jordan, Manning, and~Babai; see~Section~\ref{proofofmain3}.} 
\end{remark}

In the case of non-elemental primitive permutation groups $G$, we will make use of the following theorem relating indices of elements of  $G$ to indices of elements of the corresponding imprimitive permutation group $G'$ having the same underlying group. 

\begin{theorem}\label{nonelemental}
Let $G\subset S_m\wr S_r$ be a non-elemental primitive permutation group of type $(r,m,k)$ acting on $n={m\choose k}^r$ letters $(1\leq k<m/2)$, and let $G'$ denote the natural imprimitive action of the same underlying group on $rm$ letters.  Then for any non-identity element $g\in G$ corresponding to $g'\in G'$, we have
$$\frac{\ind(g)}{\ind(g')} > \frac n{3rm}  
\:.$$
\end{theorem}
 
 \begin{proof}
 Write $g=(g_1,\ldots,g_r,h)\in S_m^r\rtimes S_r$, where each $g_i\in S_m$ and $h\in S_r$. 
 If $h\in S_r$ is not the identity, then $g$ moves at least ${m\choose k}^{r}-{m\choose k}^{r-1}$ elements (with equality occurring when $g_1=\cdots=g_r=1$ and $h$ is a transposition); thus \smash{$\ind(g)\geq \frac12 [{m\choose k}^{r}-{m\choose k}^{r-1}]$}. Since $\ind(g')<rm$, we have 
 $$\frac{\ind(g)}{\ind(g')}> \frac{{m\choose k}^{r}-{m\choose k}^{r-1}}{2rm}\geq  \frac{{m\choose k}^{r-1}[{m\choose k}-1]}{2rm}\geq  \frac{{m\choose k}^{r-1}\frac23{m\choose k}}{2rm}
 \geq \frac{1}{3rm} {m\choose k}^{r}$$
 since $1\leq k< m/2$ and $m\geq 3$.

If $h$ is the identity, then let $i$ denote the index such that $z=\ind(g_i)$ is maximal.  
Then $g_i$ is the product of disjoint nontrivial cycles.  The element $g_i$ moves a $k$-element subset of $\{1,\ldots,m\}$ if and only if this $k$-element subset does not consist of the indices of some subset of these cycles together with fixed points of $g_i$. The number of such $k$-element sets only decreases if one breaks these cycles down further; we break each cycle of $g_i$ into disjoint 2-cycles together with at most one 1-cycle.  
We thereby obtain an element $\bar g_i$ that is the product of at least $z/2$ disjoint transpositions (with equality here when $g_i$ is the product of odd-length cycles). 

Let $f(m,k,y)$ denote the number of $k$-element subsets of $\{1,\ldots,m\}$ $(1\leq k<m/2)$ moved by an element $\sigma\in S_m$ that is the product of $y$ disjoint 2-cycles. We prove by induction that 
\begin{equation}\label{fmky}
f(m,k,y)\geq \frac8{5k} {m-1\choose k-1}\cdot y.
\end{equation}
Since $f(m,1,y)=2y\geq \frac85{m-1\choose 0}\cdot y$
and $$f(m,2,y) = {m\choose 2}-y-{m-2y\choose 2}= 2(m - y-1) y\geq \frac8{5\cdot 2}{m-1\choose 1}\cdot y,$$ the result is clearly true for $k=1$ and also for $k=2$. If $m$ and $k$ are even and $y=m/2$,~then
$$f(m,k,y)= {m\choose k} - {m/2\choose k/2}\geq \frac45 {m\choose k}\geq \frac45 \cdot \frac{2y}k{m-1\choose k-1}=\frac8{5k} {m-1\choose k-1}\cdot y$$
since $m\geq 5$ if $k$ is even. Otherwise, $\sigma$ has at least one fixed point on $\{1,\ldots,m\}$, say $m$; then, by conditioning on whether $m$ is included in the set, we see that when $k>2$ the number of $k$-subsets moved by $\sigma$ is
$$f(m,k,y)=f(m-1,k-1,y)+f(m-1,k,y)$$
and thus by the induction hypothesis
$$f(m,k,y)\geq \frac8{5(k-1)}{m-2\choose k-2}\cdot y+\frac8{5k}{m-2\choose k-1}\cdot y\geq 
\frac8{5k} {m-1\choose k-1}\cdot y,$$ as claimed.

Returning to the element $\bar g_i$ that is the product of at least $z/2$ disjoint $2$-cycles on $\{1,\ldots,m\}$, we see that such a $\bar g_i$ moves at least $\frac 8{5k}{m-1\choose k-1}\cdot z/2$ subsets of $\{1,\ldots,m\}$ of size~$k$. We conclude that $g=(g_1,\ldots,g_i,\ldots,g_r,1)$ moves at least 
$\frac 8{5k}{m-1\choose k-1}\cdot z/2 \cdot{m\choose k}^{r-1}$ subsets of size $k$, so that 
$\ind(g)\geq \frac 4{5k}{m-1\choose k-1}\cdot z/2 \cdot{m\choose k}^{r-1}$.  Since $\ind(g')\leq rz=r\cdot\ind(g_i)$, we conclude~that 
\begin{equation*}
\ind(g)\geq \frac 4{2\cdot 5rk}{m-1\choose k-1}\cdot{m\choose k}^{r-1} \ind(g')
> \frac 1{3rm}{m\choose k}^{r} \ind(g'),
\end{equation*} as desired.
 \end{proof}

\vspace{-.1in}

\begin{corollary}\label{nonebound2}
Let $G$ be a non-elemental primitive permutation group of type $(r,m,k)$ acting on $n={m\choose k}^r$ letters, and let $G'$ be the corresponding imprimitive permutation group on $rm$ letters.
Then for any non-identity element $g\in G$ corresponding to $g'\in G'$, we have 
$$\frac{\ind(g)}{\ind(g')} \gg \sqrt n.$$
\end{corollary}

\begin{proof}
Since at least one of the positive integers $r$ or $k$ is at least 2, we have that $r^2m^2\!\ll {m\choose k}^r$. Thus, by Theorem~\ref{nonelemental}, 
$$\frac{\ind(g)}{\ind(g')}\geq \frac n{3rm} \gg {n^{1/2}};$$
the exponent $1/2$ here is of course optimal and is achieved when $\{r,k\}=\{1,2\}$. 
\end{proof}

\noindent
Theorem~\ref{nonelemental} and Corollary~\ref{nonebound2} have the following implication for the discriminants of number fields having associated Galois group a non-elemental primitive permutation group $G$.

\begin{corollary}\label{nonenumfield}
Let $n={m\choose k}^r$, and let $K$ be a number field of degree $n$ with associated Galois group a non-elemental primitive permutation group $G\subset S_n$ of type $(r,m,k)$. Let $K'$ be a number field of degree $rm$ in the Galois closure $L$ of $K$ such that $G$ acts on the embeddings of $K'$ with its natural imprimitive action $G'\subset S_{rm}$, i.e., $K'$ is the $G'$-resolvent of $K$.  Then 
$$
|\Disc(K)| \gg_n |\Disc(K')|^{n/(3rm)},
$$
which is $\gg|\Disc(K')|^{c\sqrt{n}}$ for some constant $c>0$. 
\end{corollary}

\begin{proof}
If $p>n$, then $p$ is tamely ramified in $L$. Let $g\in G$ (with image $g'\in G'$) be a generator of an inertia group at $p$. Then  $v_p(\Disc(K))=\ind(g)$ and $v_p(\Disc(K'))=\ind(g')$. The result now follows from Theorem~\ref{nonelemental} and Corollary~\ref{nonebound2}. 
\end{proof}
\vspace{-.15in}

\section{Results on counting number fields}

As in the introduction, let $F_n(G,X)$ denote the number of number fields of degree n having associated Galois group G and absolute discriminant at most X. The following upper bounds for $F_n(G,X)$ that are independent of $G$ are due to Schmidt~\cite{Schmidt} and Lemke Oliver and Thorne~\cite{LOT2} (see also \cite{EV} and \cite{Couveignes}), respectively.

\begin{theorem}[Schmidt~\cite{Schmidt} and Lemke Oliver--Thorne~\cite{LOT2}]\label{slot}
We have $F_n(G,X)=$ \linebreak $O(X^{(n+2)/4})$ for any $n>1$.  There is a constant $r>0$ such that $F_n(G,X)=O(X^{r\log^2\!n})$, which yields a strictly smaller upper bound on $F_n(G,X)$ when $n\geq 95$. 
\end{theorem}

Better bounds  on $F_n(G,X)$ than those in Theorem~\ref{slot}, even when $n>5$, are known for many specific choices of permutation group~$G$ (see, e.g.,~\cite{Wright,KM,BW,LR,Dummit,Alberts,W, F}).  
Here we wish to prove improved bounds on $F_n(G,X)$ for {primitive} permutation groups $G$.

We begin with an improvement on the first assertion in Theorem~\ref{slot} for general primitive permutation groups of index $\geq 2$. 

\begin{theorem}\label{primcount}
Let $G$ be a primitive permutation group on $n$ letters.  Then $$F_n(G,X)=O_\epsilon(X^{\frac{n+2}4-1+\frac{1}{\ind(G)}+\epsilon}).$$
\end{theorem}

\noindent
We note that Theorem~\ref{primcount} improves upon the corresponding exponents for specific primitive Galois groups obtained in the works of Larson--Rolen~\cite{LR} and Dummit~\cite{Dummit} when~$n<95$. 

To prove Theorem~\ref{primcount}, we use the following proposition about polynomials having a given index $k$ over a field $L$.  A polynomial $P(x)\in L[x]$ of degree $n$ has {\it index $k$} over $L$ if it factors as $\prod P_i^{e_i}$ where $P_i$ is irreducible over $L$ of degree $f_i$ and $\sum (e_i-1) f_i=k$. 

\begin{proposition}\label{indexcodim}
Let $P(x)=x^n+a_1x^{n-1}+\cdots+a_n$ be a polynomial over a field $K$ with characteristic prime to $n!$.  If the values of $a_1,\ldots,a_{n-k}\in K$ are fixed, then there are at~most $q(k,n-k)\cdot(n-k)!$ values of $a_{n-k+1},\ldots,a_n\in K$ such that $P(x)$ has index $k$ over $K$, where~$q(k,n-k)$ denotes the number of partitions of $k$ into at most $n-k$ parts. 
\end{proposition}

\noindent
Proposition~\ref{indexcodim} is proven in turn using the following proposition on weighted sums of powers.

\pagebreak 

\begin{proposition}\label{genvan}
Let $m_1,\ldots,m_r$ be constants in a field $L$ such that no nonempty subset of $\{m_1,\ldots,m_r\}$ sums to zero.  Let $c_1,\ldots,c_r\in L$ be arbitrary.  Then the system of equations 
\begin{equation}\label{psystem}
\begin{array}{rl}
m_1 x_1 + m_2 x_2 + \cdots + m_r x_r &= c_1 \\[.045in]
m_1 x_1^2 + m_2 x_2^2 + \cdots + m_r x_r^2 &= c_2 \\[.05in]
m_1 x_1^3 + m_2 x_2^3 + \cdots + m_r x_r^3 &= c_3 \\[.00in]
 \hspace{-1in}\vdots\hspace{.832in}& \\[.03in]
 m_1 x_1^r + m_2 x_2^r + \cdots + m_r x_r^r &= c_r 
\end{array}
\end{equation}
has only finitely many $($in fact, at most $r!)$ solutions in $L$.
\end{proposition}

\begin{proof}
We first prove that if $c_1=\cdots=c_r=0$, then $x_1=\cdots=x_r=0$ is the only solution to the system of equations (\ref{psystem}).  We proceed by induction, the claim being trivial when $r=1$. For general $r$, we view (\ref{psystem}) as a system of linear equations in $m_1,\ldots,m_r$. The determinant of the resulting coefficient matrix, by the evaluation of a Vandermonde determinant, is equal to $D=\prod_i x_i\prod_{i<j} (x_i-x_j)$.  If $D\neq 0$, then the only solution to the linear system (\ref{psystem}) is $m_1=\cdots=m_r=0$, contradicting the assumption that no nonempty subset of $\{m_1,\ldots,m_r\}$ sums to zero. Hence if (\ref{psystem}) has a solution for $x_1,\ldots,x_r$, then some $x_i$ or some $x_i-x_j$ must vanish.  If some $x_i$, say $x_r$, vanishes, then the first $r-1$ equations in (\ref{psystem}) in the variables $x_1,\ldots,x_{r-1}$ satisfy the induction hypothesis for $r-1$, since no nonempty subset of of $\{m_1,\ldots,m_{r-1}\}$ will sum to zero either; hence, in this case, $x_1=\cdots=x_{r-1}=0$ (implying $x_r=0$) is the only solution to (\ref{psystem}). Similarly, if some $x_i-x_j$, say $x_{r-1}-x_r$, vanishes, then again 
 the first $r-1$ equations in (\ref{psystem}) (viewed as equations in the variables $x_1,\ldots,x_{r-1}$) satisfy the induction hypothesis for $r-1$, since again no nonempty subset of the corresponding constants $m_1,\ldots,m_{r-1}+m_r$ will sum to zero.  We conclude that  $x_1=\cdots=x_r=0$ is the only solution to the system of equations (\ref{psystem}) if $c_1=\cdots=c_r=0$ and no nonempty subset  of $\{m_1,\ldots,m_r\}$ sums to zero.
 
We consider next the case where $c_1,\ldots,c_r\in L$ are general constants. Suppose that~(\ref{psystem}) has infinitely many solutions for $x_1,\ldots,x_r$ over $L$, or even over an algebraic closure $\overline{L}$ of~$L$. This implies that the $r$ homogeneous expressions $h_j :=\sum_i m_i x_i^j- c_j x_{r+1}^j$ ($j=1,\ldots,r$) in $x_1,\ldots,x_{r+1}$ define a projective variety of dimension at least 1 in $\P^r$. The intersection of this projective variety of positive dimension with the hyperplane $x_{r+1}=0$ thus contains at least one point in $\P^{r-1}(\overline{L})$, implying that the system (\ref{psystem}) has a solution in $\P^{r-1}(\overline{L})$ when $c_1=\cdots=c_r=0$, a contradiction. 
We conclude that, for any choice of $c_1,\ldots,c_r$, the system~(\ref{psystem}) has finitely many solutions for $x_1,\ldots,x_r$.

In particular, when not all the $c_i$ are zero, the $h_i$ define a zero-dimensional complete intersection in $\P^{r}$; by Bezout's Theorem, the number of solutions for $x_1,\ldots,x_r$ in this case is thus at most $r!$. 
\end{proof}

\vspace{-.15in}

\begin{remark}\label{irremark}
{\em Let $L$ be an algebraically-closed field such that $\char(L)\nmid n!$. Then the irreducible components of the variety of monic polynomials of degree~$n$ having index~$\geq k$ over~$L$ are given by the partitions $(m_1,\ldots,m_{n-k})$ of $n$ (where, as usual, $m_1\geq m_2\geq \cdots \geq m_{n-k}$).  Indeed, given such a partition, we have a morphism from $\mathbb{A}^{n-k}$ onto a component of the variety of monic polynomials of index~$k$, given by $(r_1,\ldots,r_{n-k})\mapsto (x-r_1)^{m_1}\cdots(x-r_{n-k})^{m_{n-k}}$. The image of this morphism is irreducible, being a surjective image of affine space. 

We claim that the image of $\A^{n-k}$ in the space of monic polynomials of degree $n$ under this morphism does not lie in any hyperplane.  Indeed, suppose $f(x)$ is in the image. Then so is $f(t(x+c))$ for any $c,t\in L$.  Choose $c\in L$ so that all coefficients of $f(x+c)$ are nonzero; this condition on $c$ is equivalent to the condition that $c$ is not a root of any of the nonzero polynomials $f(x)$, $f'(x)$, $\ldots\,$, $f^{(n-1)}(x)$ of degrees $n$, $n-1$, $\ldots$\,, $1$, respectively. 
Now let $t$ vary in $L$; then the set of polynomials $f(t(x+c))$ defines a rational normal curve in $\A^n$, and hence does not lie in any hyperplane. 

The irreducible components of the variety of binary $n$-ic forms over $L$ having index~$k$ is described in the same way, but where each irreducible component is now the image of a product of projective lines instead of affine lines, since roots of binary $n$-ic forms lie in $\P^1$ rather than $\A^1$.  Each of these irreducible components again does not lie in any hyperplane, for essentially the same reason: given an $f$ in one of these components, we choose a $\GL_2(L)$-transformation of $f$ to make its leading coefficient $1$, and then dehomogenize to reduce to the monic case. }
\end{remark}

{\noindent {\bf Proof of Proposition~\ref{indexcodim}:}}
Suppose $P(x)=x^n+a_1x^{n-1}+\cdots+a_n$ is a polynomial over~$K$ of index $k$ having distinct roots $x_1,\ldots,x_r$, over an algebraic closure $\overline{K}$ of $K$, with respective multiplicities $m_1,\ldots,m_r$. Then the condition that $P$ has index $k$ is~equivalent to the equality $r=n-k$, and the number of possibilities for the $r$-tuple $(m_1,\ldots,m_r)$ of multiplicities with $m_1\geq m_2\geq \cdots \geq m_r$ is the number of partitions of $n$ into $n-k$ parts or, equivalently, the number $q(k,n-k)$ of partitions of $k$ into at most $n-k$ parts.  

Fix such a choice $(m_1,\ldots,m_r)$. Specifying the coefficients $a_1,\ldots,a_{r}$  of the polynomial $P(x)=x^n+a_1x^{n-1}+\cdots+a_n$ is~equivalent to specifying the first $r$ elementary symmetric functions of the $n$ roots of $P(x)$ in $\overline{K}$. By the Girard--Newton identities, this in turn is equivalent to specifying the first $r$ power sums $c_1,\ldots,c_{r}$ of the $n$ roots of~$P(x)$ in $\overline{K}$, where $c_1,\ldots,c_r$ are certain fixed integer polynomials in $a_1,\ldots,a_r$. Thus, if $a_1,\ldots,a_r\in K$ (equivalently, $c_1,\ldots,c_r\in K$) are specified, and $P(x)$ has distinct roots $x_1,\ldots,x_r$ over~$\overline{K}$ with~respective multiplicities $m_1,\ldots,m_r$, then $x_1,\ldots,x_r$ must satisfy the equations given in~(\ref{psystem}). No nonempty subset of $\{m_1,\ldots,m_r\}$ sums to zero in $K$ since the integers $m_i$ are positive and sum to $n$, and $\char(K)\nmid n!$. By Proposition~\ref{genvan}, we conclude that, if $(m_1,\ldots,m_r)$ and $a_1,\ldots,a_r$ are fixed, then there are at most~$r!$ possibilities for $x_1,\ldots,x_r$~in $\overline{K}$, and hence at most $r!$ possibilities for $a_{r+1},\ldots,a_n\in K$. 

Therefore, if the values of $a_1,\ldots,a_{n-k}\in K$ are fixed, then there are  at~most \linebreak $q(k,n-k)\cdot(n-k)!$ values of $a_{n-k+1},\ldots,a_n\in K$ such that $P(x)$ has index~$k$ over~$K$.~{$\Box$ \vspace{2 ex}}

To prove Theorem~\ref{primcount}, we also make use of the following lemma relating the discriminant of an integer polynomial to its index over $\F_p$.

\begin{lemma}\label{discindex}
Let $g(x)\in\Z[x]$ be a polynomial such that the $\Q$-algebra $\Q[x]/(g(x))$ has nonzero discriminant $D$. Then for each prime~$p>n$, the index of $g(x)$ over $\F_p$ is at least $v_p(D)$.
\end{lemma}

\begin{proof}
Let $p>n$ be a prime. We work locally over $\Q_p$ and $\Z_p$. For a separable polynomial $h$ over $\Q_p$, let $L_h$ denote the \'etale $\Q_p$-algebra $\Q_p[x]/(h(x))$. If $h(x)=h_1(x)h_2(x)$ for some polynomials $h,h_1,h_2$ over $\Z_p$, then the index of $h$ over $\F_p$ is at least the sum of the indices of $h_1$ and $h_2$ over $\F_p$. Meanwhile, $v_p(\Disc(L_h))=v_p(\Disc(L_{h_1}))+v_p(\Disc(L_{h_2}))$. Thus, to prove the lemma, we may reduce to the case where the polynomial $g(x)\in\Z_p[x]$ is irreducible over $\Z_p$, i.e., $L_g$ is a field extension of $\Q_p$. 

Since $p>n$, the field extension $L_g$ of $\Q_p$ is tamely ramified, and so $v_p(D)=ef$ where~$e$ is the ramification index and~$f$ is the residue field degree of $L_g$. We wish to show that the index of $g$ over $\F_p$ is at least $ef$. By the definition of index, the index of $g$ over $\F_p$ is equal to the index of $g$ over~$\F_q$ for any extension $\F_q$ of $\F_p$. We thus bound instead the index of $g$ over $\F_q$ from below, where $q=p^f$.  

Let $\Q_q$ denote the unramified extension of $\Q_p$ of degree $f$. Then $L_g\cong \Q_q[\pi]$, where $\pi$ is a uniformizer of $L_g$ and $(p)=(\pi^e)$ as ideals of the valuation ring of $L_g$. Therefore, $\Q_q[x]/(g(x))\cong L_g\otimes_{\Q_p} \Q_q\cong \prod_{i=1}^f \Q_q[\pi]$ as $\Q_q$-algebras. Hence, over $\Q_q$, we have a factorization $g=\prod_{i=1}^f g_i$, where $\Q_q[x]/(g_i)\cong L_g$ for each $i$. 

Let $\alpha$ denote the image of $x$ in $L_g\cong \Q_q[x]/(g_i)$. Let $c\in\Q_q$ be an element such that $\alpha- c$ has positive valuation. Then $g_i(x-c)$ is the minimal polynomial over $\Q_q$ of an element of positive valuation in $L_g$, and so $g_i(x-c)\equiv x^{e}$ (mod $p$) as polynomials over $\F_q$, i.e., $g_i$ has index $e$ over $\F_q$. As this holds for every $i$, we conclude that $g=\prod_{i=1}^f g_i$ has index at least $ef$ over $\F_q$ and thus over $\F_p$, as desired. 
\end{proof}

{\noindent {\bf Proof of Theorem~\ref{primcount}:}}
Theorem~\ref{primcount} follows from Theorem~\ref{slot} when $\ind(G)=1$, and so we may assume that $\ind(G)\geq 2$.  

Now suppose that $K$ is a $G$-number field of degree $n$ with ring of integers~$\O_K$ and absolute discriminant $D$. Let $\alpha\in \O_K-\Z$ be the element of trace zero such that $||\alpha||$, the square root of the sum of the squares of the absolute values of the archimedean embeddings of $\alpha$, is minimal. Then, as in~\cite{Schmidt}, we have $||\alpha||\ll D^{1/(2n-2)}$. The minimal polynomial $P(x)=x^n+a_2x^{n-2}+\cdots+a_n$ of $\alpha$ thus satisfies $|a_i|\ll D^{i/(2n-2)}$.

For every $p\mid D$ and $p>n$, the index $k_p$ of $P(x)$ over $\F_p$ is at least~$v_p(D)$ by Lemma~\ref{discindex}, 
and $v_p(D)$ in turn is at least $k:=\ind(G)$ by Proposition~\ref{valindcon}. 
For any such $p$, by Proposition~\ref{indexcodim}, the values of $a_{n-v_p(D)+1},\ldots,a_{n}$ modulo~$p$ are determined up to a bounded number $C$ of possibilities by $a_1,\ldots,a_{n-v_p(D)}$.  Thus $a_{n-j+1}$ is determined modulo $D_j:=\prod_{p:v_p(D)\geq j}p$ up to at most $C^{\omega(D_j)}=O_\epsilon(D^\epsilon)$ possibilities by the Chinese Remainder Theorem.  By definition, we have that $D_j=O(D^{1/j})$, while the range of $a_{n-j+1}$ is $\ll D^{(n-j+1)/(2n-2)}$. 

Now, we observe that the inequality $1/j\leq (n-j+1)/(2n-2)$ holds for all $j\in [2,n-1]$, with equality when $j=2$ or $j=n-1$, as the latter inequality is equivalent to $(j-2)(j-n+1)\leq 0$. Since $\ind(G)\geq 2$, it follows that $D_1=D_2\ll D^{n/(2n-2)}$ as well. Therefore, the inequality $$D_j\ll D^{(n-j+1)/(2n-2)}$$ holds for all $j\in\{1,\ldots,n-1\}$.

It follows that if $a_2,\ldots,a_{n-j}$ are fixed, then there are $O(C^{\omega(D_j)}D^{(n-j+1)/(2n-2)}/D_j)$ 
possibilities for $a_{n-j+1}$. Hence the total number of possible minimal polynomials for $\alpha$ is 
$$\ll \prod_{i=2}^n\frac{C^{\omega(D)}D^{\frac{i}{2n-2}}}{D_{n+1-i}}=O_\epsilon\!\left(\frac{D^{\frac{n+2}4+\epsilon}}{D}\right)=O_\epsilon({D^{\frac{n+2}4-1+\epsilon}});$$
this therefore represents an upper bound on the total number of $G$-number fields $K$ having degree~$n$ and given discriminant $D$. 

Summing this bound over all $O(X^{1/k})$ discriminants $D$ less than $X$ that are $k$-powerful now yields the stated bound on the total number of $G$-number fields~$K$ of degree~$n$ having absolute discriminant $D<X$. {$\Box$ \vspace{2 ex}}

We turn next specifically to {\bf non-elemental} primitive permutation groups, where we obtain significantly stronger bounds than those in Theorem~\ref{slot} and Theorem~\ref{primcount}. 

\begin{theorem}\label{nonecount}
Let $G\subset S_m\wr S_r$ be a non-elemental primitive permutation group of type $(r,m,k)$ acting on $n={m\choose k}^r$ letters. Then 
$$F_n(G,X)=O(X^{brm\log^2(rm)/n}
)=O(X^{b\log^2\!n/\!\sqrt{n}})$$ for some constant $b>0$. 
\end{theorem}

\begin{proof}
Let $G'$ denote the associated imprimitive permutation group, with the same underlying group as $G$ but acting on $rm$ letters. By Theorem~\ref{slot}, the number of degree~$rm$ extensions~$K'$ with associated Galois group $G'$ and absolute discriminant at most $Y$ is $O_n(Y^{c\log^2(rm)})$ for an absolute constant $c>0$.  In the Galois closure of~$K'$ is an extension $K$ of degree~$n={m\choose k}^r$ having associated Galois group $G$ (i.e., $K$ is the $G$-resolvent of~$K'$).  By Corollary~\ref{nonenumfield}, we have $|\Disc(K')|\ll_n |\Disc(K)|^{3rm/n
}.$
Since every $K$ arises from a~$K'$, and the number of isomorphism classes of $K$ that can arise from an isomorphism class of $K'$ in this way is clearly bounded in terms of $n$ (e.g., by the number of distinct partitions of $rm$ letters that are preserved by $G'$), we see that $F_n(G,X)\ll_n F_{rm}(G',X^{3rm/n})$.  By~Schmidt's bound as in the first part of Theorem~\ref{slot}, we thus obtain
\begin{equation}\label{nonecount2}
F_n(G,X)\ll_n F_{rm}(G',X^{3rm/n})=O((X^{3rm/n})^{(rm+2)/4}) = O(X^{3rm(rm+2)/(4n)})=O(X^{4}),  
\end{equation}
which serves as a useful estimate for smaller values of $n$; for general large $n$, using Lemke Oliver and Thorne's bound as in the second part of Theorem~\ref{slot}, we obtain
$$F_n(G,X)\ll_n F_{rm}(G',X^{3rm/n})=O((X^{3rm/n})^{b\log^2(rm)}),$$
which in turn is $O(X^{b\log^2\!n/\!\sqrt{n}})$ 
by Corollary~\ref{nonebound2}, for some constant $b>0$. 
\end{proof}

\vspace{-.15in}

\section{Equidistribution modulo $C$ of polynomials cutting out field extensions ramified at primes dividing~$C$}\label{Fourier}

A common ingredient in proving that sets are equidistributed modulo a prime $p$ is Fourier analysis over {$\F_p$} (see, e.g., Taniguchi--Thorne~\cite{TT} for a nice introduction to this topic where, in particular,  the case of binary cubic forms is treated). Here, we apply Fourier analysis to show that integral binary forms (resp.\ monic polynomials) of degree $n$ 
that have index at least a given integer $k$ over $\F_p$ are well distributed  in boxes.  More generally, for a squarefree positive integer $q$, we show that integral binary forms (resp.\ monic polynomials) of degree~$n$, that have indices at least a given integer $k_p$ over $\F_p$ for each prime $p\mid q$, are also well distributed in boxes. 

\subsection{Binary forms}

For a ring $R$, let $V_R$ denote the $R$-module of rank $n+1$ consisting of all binary $n$-ic forms having coefficients in $R$.  Let $q$ be any positive integer.  For a function $\Psi_q: V_{\Z/q\Z} \rightarrow \C$, let \smash{$\widehat{\Psi_q}\colon V^\ast_{\Z/q\Z} \rightarrow \C$} be its Fourier transform defined by the usual formula
\begin{equation}\label{fformula}
{\widehat{\Psi_q}(g) = \frac1{q^{n+1}}
\sum_{f \in V_{\Z/q\Z}} \Psi_q(f) \exp\left(\frac{2 \pi i [f, g]}{q}\right),}
\end{equation}
where $[f,g]$ denotes the canonical bilinear form \smash{$V^\ast_{\Z/q\Z}\times V_{\Z/q\Z}\to\Z/q\Z$}. 
When $\Psi_q$ is the characteristic function of a set {$S\subset V_{\Z/q\Z}$}, then the principle is that upper bounds on the maximum $M(\Psi_q)$ of \smash{$|\widehat{\Psi}_q(g)|$} over all nonzero $g\in V^\ast_{\Z/q\Z}$ yields a measure of 
{equidistribution} of {$S$} in a box \smash{$[-H,H]^{n+1}$} of integer-coefficient binary $n$-ic forms.  The reason is that, for any Schwartz function {$\phi$} approximating the characteristic function of the box $[-1,1]^{n+1}$, the twisted Poisson summation formula states
\begin{equation}\label{twisted}
\sum_{f \in V_\Z} \Psi_q(f) \phi(f/H)  = 
H^{n+1} \sum_{g \in V^*_\Z} \widehat{\Psi_q}(g) \widehat{\phi}\left(\frac{g H}{q} \right).
\end{equation}
For suitable {$\phi$}, the left side will be an upper bound for the number of 
elements in $S$ in the box $[-H,H]^{n+1}$. The {$g = 0$} term is the expected main term; meanwhile,
the rapid decay of \smash{$\widehat{\phi}$} implies that the error term is effectively bounded by $H^{n+1}$ times the sum of \smash{$|\widehat{\Psi_q}(g)|$} 
over all $0\neq g\in V^*_\Z$ whose coordinates are of size at most $O_\epsilon(q^{1+\epsilon}/H)$, and this in turn can be bounded by $O(H^{n+1}(q^{1+\epsilon}/H)^{n+1}M(\Psi_q))=O(q^{n+1+\epsilon}M(\Psi_q))$.

Let us now consider the factorization behavior of binary forms modulo a prime $p$. If a binary $n$-ic form $f$ (over $\Z$, or over $\F_p$) factors modulo $p$ as $\prod_{i=1}^r P_i^{e_i}$, with $P_i$ irreducible and $\deg(P_i)=f_i$, then the {\it splitting type} $(f,p)$ of $f$ is defined as $(f_1^{e_1}\cdots f_r^{e_r})$, and the {\it index} $\ind(f)$ of $f$ modulo~$p$ (or the {\it index} of the splitting type $(f,p)$ of $f$) is defined to be $\sum_{i=1}^r (e_i-1)f_i$. 

More abstractly, we call any expression $\sigma$ of the form $(f_1^{e_1}\cdots f_r^{e_r})$ a {\it splitting type}; the {\it degree} $\deg(\sigma)$ is $\sum_{i=1}^r e_i f_i$, and the {\it index} $\ind(\sigma)$ is $\sum_{i=1}^r (e_i-1)f_i$. The {\it length} $\len(\sigma)$ is $\deg(\sigma)-\ind(\sigma)=\sum_{i=1}^r f_i$. Finally, $\#\Aut(\sigma)$ is defined to be $\prod_i f_i$ times the number of permutations of the factors
$f_i^{e_i}$ that preserve $\sigma$,\footnote{See~\cite[\S2]{mass} for the motivation for this definition.} e.g., $\#\Aut(112^52^52^53^{11})=24\cdot 2\cdot6=288.$ 

  In this section, we first show that the set of binary $n$-ic forms over $\F_p$ having given index---and more generally, having splitting type containing a given splitting type $\sigma$---are very well distributed in $[-H,H]^{n+1}$ for suitable values of $H<p$. We will accomplish this, as described above, via demonstrating cancellation in the Fourier transforms of the corresponding characteristic functions.

\begin{proposition}\label{ftgen}
Let $\sigma=(f_1^{e_1}\cdots f_r^{e_r})$ 
be a splitting type with $\deg(\sigma)=d$, $\ind(\sigma)=k$, and $\len(\sigma)=\ell$. 
Let $w_{p,\sigma}:V_{\F_p}\to\C$ be defined by
\begin{align*}
w_{p,\sigma}(f):=&\text{ the number of $r$-tuples $(P_1,\ldots,P_r)$, up to the action of the group of}  \\[-.04in] & \text{  permutations of $\{1,\ldots,r\}$ preserving $\sigma$, such that the $P_i$ are distinct  } \\[-.04in] &\text{ irreducible binary forms where, for each $i$, we have $P_i(x,y)$ is $y$ or is  } \\[-.04in] &\text{ monic as a polynomial in $x$, $\deg P_i=f_i$, and $P_1^{e_1}\cdots P_r^{e_r}\mid f$}.
\end{align*}
Then 
\begin{equation*}
\widehat{w_{p,\sigma}}(g)=
\begin{cases}
	{\displaystyle\frac{p^{-k}}{\small\Aut(\sigma)}\!+\!O(p^{-(k+1)})}	& \!\text{if } {g=0};\\[.1in]
	{O(p^{-(k+1)})}		& \!\text{if $g \neq 0$}.\\
\end{cases}
\end{equation*}
\end{proposition}

\begin{proof}
We have
\begin{eqnarray}\displaystyle
\widehat{w_{p,\sigma}}(g)&=&\displaystyle\frac{1}{p^{n+1}}\sum_{f\in V_{\F_p}} e^{2\pi i[f,g]/p} w_{p,\sigma}(f)
\label{firstexp}\\
&=&\displaystyle\frac{1}{p^{n+1}}\sum_{P_1,\ldots,P_r}
\:\sum_{f\,:\,P_1^{e_1}\cdots P_r^{e_r} \mid f}
e^{2\pi i[f,g]/p}. \label{secondexp}
\end{eqnarray}

Evaluating (\ref{secondexp}) for $g=0$ gives 
$\widehat{w_{p,\sigma}}(0)=(p^{-k}/\#\Aut(\sigma))+O(p^{-(k+1)})$ 
because the number of possibilities for $P_1,\ldots,P_r$ is 
$(1/\#\Aut(\sigma))p^{\ell}+O(p^{{\ell-1}})$, and the set of $f$ such that $P_1^{e_1}\cdots P_r^{e_r}$ divides $f$ is a linear subspace $U\subset V_{\F_p}$ of codimension $d=\deg(P_1^{e_1}\cdots P_r^{e_r})$; hence the inner sum has $p^{n+1-d}$ terms, and therefore 
$$\widehat{w_{p,\sigma}}(0)=\frac{1}{p^{n+1}}\left(\frac{p^{\ell}}{\#\Aut(\sigma))}+O(p^{{\ell-1}})\right)p^{n+1-d} = {\displaystyle\frac{p^{-k}}{\small\Aut(\sigma)}\!+\!O(p^{-(k+1)})}.	$$

Next, assume $g\neq 0$. 
We consider each of the $(1/\#\Aut(\sigma))p^{\ell}+O(p^{{\ell-1}})$  terms in the outer sum separately. 
We thus fix $P_1,\ldots,P_r$.

Let $U:=U_{P_1^{e_1},\ldots,P_r^{e_r}}\subset V_{\F_p}$ denote the codimension $d$ linear subspace of all $f\in V_{\F_p}$ such that  $P_1^{e_1}\cdots P_r^{e_r}$ divides $f$. 
The set of $g\in V^*_{\F_p}$ such that $[f,g]=0$ for all $f\in U$ is thus a linear subspace $W:=W_{P_1^{e_1},\ldots,P_r^{e_r}}\subset V^*_{\F_p}$ of dimension $d$. 
If $g\in W$, the contribution of the inner sum for this fixed choice of $P_1,\ldots,P_r$ is thus $p^{n+1-d}$. Meanwhile, if $g\notin W$, then $\{f\in U: [f,g]=0\}$ is codimension~1 subspace of~$U$, yielding a union of $p^{n-d}$ complete exponential sums modulo $p$, and so the contribution of the inner sum for such $g\notin W$ is~$0$. 

We claim that the number of choices of $P_1,\ldots,P_r$ for which $g\in W_{P_1^{e_1},\ldots,P_r^{e_r}}$ is~$O(p^{\ell-1})$.  Indeed, the subspace $U_{P_1^{e_1},\ldots,P_r^{e_r}}\subset V_{\F_p}^\ast$ has cardinality $p^{n+1-d}$ and determines $P_1^{e_1}\cdots P_r^{e_r}$.  The hyperplane $\Ann(g)\subset V_{\F_p}$ annihilated by $g$ has dimension~$n$; hence the number of polynomials having index at least $k$ in $\Ann(g)$ is $O(p^{n-k})$, since when $p>n$ none of the irreducible components of the variety in $V$ of codimension $\geq k$, consisting of polynomials having index~$\geq k$, lie in any hyperplane in~$V$ (see Remark~\ref{irremark}). Because an element of $V_{\F_p}$ can lie in at most $O_n(1)$ $U_{P_1^{e_1},\ldots,P_r^{e_r}}$'s, it follows that the maximum number of possibilities for $P_1^{e_1},\ldots,P_r^{e_r}$, such that $g$ annihilates all multiples of $P_1^{e_1}\cdots P_r^{e_r}$ in $V_{\F_p}$, is $O(p^{n-k}/p^{n+1-d})=O(p^{\ell-1})$.

We conclude that
$$\widehat{w_{p,\sigma}}(g)=\frac1{p^{n+1}}\cdot p^{n+1-d}\cdot O(p^{\ell-1}) = O(p^{-(k+1)}),	$$
as desired.
\end{proof}


\begin{corollary}\label{index2forms}
Let $p$ be a prime.  The number of integer-coefficient binary $n$-ic forms 
in $[-H,H]^{n+1}$ 
that, modulo $p$, have index at least $k$ is at most $O_\epsilon(H^{n+1}/p^k+p^{n-k+\epsilon})$. 
\end{corollary}

\begin{proof}
Let $\phi$ be a smooth function with compact support that is identically $1$ on $[-1,1]^{n+1}$.  Let $\Psi:V_{\F_p}\to \R$ be defined by $\Psi=\sum_{\sigma:\ind(\sigma)\geq k}w_{p,\sigma}$. By Proposition~\ref{ftgen}, $M(\Psi)=O(p^{-(k+1)})$. By twisted Poisson summation~(\ref{twisted}), we have
\begin{eqnarray}
&\!\!\!\!\!\! & 
\sum_{f \in V_\Z} \Psi(f) \phi(f/H)\\  &\!\!\!\!=\!\!\!\!& H^{n+1} \sum_{g \in V^*_\Z} \widehat{\Psi}(g) \widehat{\phi}\left(\frac{g H}{p} \right) \\
&\!\!\!\!\ll\!\!\!\!& H^{n+1}\widehat{\Psi}(0) \widehat{\phi}(0) + H^{n+1}\!\hspace{-.2in} \sum_{g \in \bigl[-\frac{p^{1+\epsilon}}{H},\frac{p^{1+\epsilon}}{H}\bigr]^{n+1}\cap V^*_\Z\setminus\{0\}} \hspace{-.2in} \!\!\left|\widehat{\Psi}(g)\right|+ H^{n+1} \hspace{-.2in}
\sum_{g\notin \bigl[-\frac{p^{1+\epsilon}}{H},\frac{p^{1+\epsilon}}{H}\bigr]^{n+1}\cap V^*_\Z}\hspace{-.2in}\left|\widehat{\phi}\left(\frac{g H}{p} \right)\right| \\[.1in]
&\!\!\!\!\ll_{\epsilon,N}\!\!\!\!& H^{n+1}/p^k + 
H^{n+1}(p^{1+\epsilon}/H)^{n+1} p^{-(k+1)} 
+ H^{n+1}  \!\!\!\!\!\!\!\sum_{g\notin  \bigl[-\frac{p^{1+\epsilon}}{H},\frac{p^{1+\epsilon}}{H}\bigr]^{n+1}\cap V^*_\Z} \hspace{-.2in}\left(\frac{\|g\| H}{p} \right)^{-N} 
\end{eqnarray}
for any integer $N$, where the bound on the third summand follows since $\phi$ is smooth and thus is $N$-differentiable for any integer $N$, and as a consequence, $\widehat{\phi}(g)\ll_N \|g\|^{-N}$ (see, e.g., \cite[Chapter~5 (Theorem~1.3)]{SS}). 
Choosing $N$ sufficiently large, e.g., $N > n+1+k/\epsilon$, yields the desired result. 
\end{proof}

\vspace{-.1in}
\begin{corollary}\label{Dequi}
Let $D$ be a positive integer with prime factorization $D=p_1^{k_1}\cdots p_m^{k_m}$ and let 
$C=p_1\cdots p_m$. The number of integer-coefficient binary $n$-ic forms 
in $[-H,H]^{n+1}$ 
that, modulo $p_i$, have index at least $k_i$ for $i=1,\ldots,m$ is at most $O_\epsilon(c^{\omega(C)}H^{n+1}/D+C^\epsilon\prod_{i=1}^mp_i^{n-k_i})$ for some constant $c>0$ dependent only on $n$.
\end{corollary}

\begin{proof}
First, note that the values of the $\Z/C\Z$-Fourier transform are simply products of values of the $\F_{p_i}$-Fourier transforms (one value for each $i$). 

Let $\phi$ again be a smooth function with compact support that is identically $1$ on $[-1,1]^{n}$.  Let $\Psi:V^1_{\Z/C\Z}\to \R$ be defined by $\Psi=\prod_i(\sum_{\sigma:\ind(\sigma)\geq k}w_{p_i,\sigma}$).
Then we have exactly as in the case $m=1$ in Theorem~\ref{index2forms} that 
\begin{eqnarray}
& & 
\sum_{f \in V_\Z} \Psi(f) \phi(f/H) \\ &\!\!\!\!\!\ll_{\epsilon,N}\!\!\!\!& c^{\omega(C)}H^{n+1}/D +
H^{n+1}\hspace{-.2in} \!\sum_{g \in \bigl[-\frac{C^{1+\epsilon}}{H},\frac{C^{1+\epsilon}}{H}\bigr]^{n+1}\cap V^*_\Z\setminus\{0\}} \!\hspace{-.2in} \!\!\left|\widehat{\Psi}(g)\right|
+ H^{n+1} \!\!\!\!\!\!\!\!\sum_{g\notin  \bigl[-\frac{C^{1+\epsilon}}{H},\frac{C^{1+\epsilon}}{H}\bigr]^{n+1}\cap V^*_\Z} \!\hspace{-.2in} \left(\frac{\|g\| H}{C} \right)^{-N} 
\end{eqnarray}
where as before the last term on the right hand side can be absorbed into the first term by fixing $N$ to be sufficiently large. 
We now estimate the second term:
\begin{equation*}
\begin{array}{rcl}
\displaystyle H^{n+1}\hspace{-.175in}  \sum_{g\in\bigl[-\frac{C^{1+\epsilon}}{H},\frac{C^{1+\epsilon}}{H}\bigr]\cap V_\Z^*\backslash\{0\}} \hspace{-.2in} \left|\widehat{\Psi}(g)\right|
&\!\ll\!&\displaystyle
H^{n+1}c^{\omega(C)}\sum_{\substack{q\mid C\\q\neq C}}\sum_{\substack{g\in\bigl[-\frac{C^{1+\epsilon}}{H},\frac{C^{1+\epsilon}}{H}\bigr]\cap V_\Z^*\backslash \{0\}\\(\ct(g),C)=q}}q\prod_{i=1}^m p_i^{-k_i-1}
\\[.425in]
&\!\ll_\epsilon\!&\displaystyle
H^{n+1}\sum_{\substack{q\mid C\\q\neq C}}
\frac{C^{n+1+\epsilon}}{q^{n+1}H^{n+1}}\cdot q\prod_{i=1}^m p_i^{-k_i-1}
\\[.375in]
&\!\ll_\epsilon\!&\displaystyle
C^\epsilon \prod_{i=1}^m p_i^{n-k_i},
\end{array}
\end{equation*}
where $\ct(g)$ denotes the gcd of the coordinates of $g$.  This yields the desired result.
\end{proof}

\vspace{-.1in}

\begin{corollary}\label{fbound}
Let $0<\delta<1/n$. Let $D$ be a positive integer with prime factorization $D=p_1^{k_1}\cdots p_m^{k_m}$ such that 
$C=p_1\cdots p_m<H^{1+\delta}$. 
Then the number of integer-coefficient binary $n$-ic forms 
in $[-H,H]^{n+1}$ that, modulo~$p_i$, have index at least $k_i$, is $O(c^{\omega(C)} H^{n+1}/D)$
for~some~constant $c>0$ dependent only on $n$.
\end{corollary}

\begin{proof}
By  Corollary~\ref{Dequi}, the number of integer-coefficient binary $n$-ic forms with the desired property is $O_\epsilon(c^{\omega(C)}H^{n+1}/D+C^{\epsilon}\prod_{i=1}^mp_i^{n}/\prod_{i=1}^mp_i^{k_i})=
O(c^{\omega(C)} H^{n+1}/D).$
\end{proof}

\vspace{-.1in}

\subsection{Monic polynomials}\label{Fourier2}

For a ring $R$, let $V^1_R\subset V_R$ denote the subset of binary $n$-ic forms over $R$ having leading coefficient 1, which we identify with the space of monic polynomials of degree $n$ over $R$.  For a function $\Psi_q : V_{\Z/q\Z}^1 \rightarrow \C$, let \smash{$\widehat{\Psi_q}\colon V^{1\ast}_{\Z/q\Z} \rightarrow \C$} be its Fourier transform defined by the usual formula
$$
{\smash{\widehat{\Psi_q}(g)} = \frac1{q^{n}}
\sum_{f \in V^{1}_{\Z/q\Z}} \Psi_q(f) \exp\left(\frac{2 \pi i [f, g]}{q}\right).}
$$
When $\Psi_q$ is the characteristic function of a set {$S\subset V^1_{\Z/q\Z}$}, then upper bounds on the maximum of \smash{$\widehat{\Psi}_q(g)$} over all nonzero $g$ constitutes a measure of 
{equidistribution} of {$S$} in suitable boxes of integer-coefficient binary $n$-ic forms since, for any Schwartz function {$\phi$} approximating the characteristic function of the box $[-1,1]^{n}$, the twisted Poisson summation formula gives 
{\begin{equation}
\begin{array}{rl}\label{twist2}
& \displaystyle\sum_{f=(a_1,\ldots,a_n) \in V_\Z^1} \Psi_q(a_1,\ldots,a_n) \phi(a_1/H,a_2/H,\ldots,a_n/H) \\[.25in] = 
H^{n} &\displaystyle\sum_{g=(b_1,\ldots,b_n) \in V^{1*}_\Z} \widehat{\Psi_q}\,(b_1,\ldots,b_n) \widehat{\phi}\left(b_1H/q,b_2H/q\ldots,b_nH/q\right).
\end{array}\end{equation}}For suitable {$\phi$}, the left side of (\ref{twist2}) will be an upper bound for the number of 
elements in~$S$ in the box $[-H,H]^{n}$. The  {$g = 0$} term is the expected main term,  while the rapid decay of \smash{$\widehat{\phi}$} implies that the error term is effectively bounded by $H^{n}$ times the maximum of \smash{$|\widehat{\Psi_q}(g)|$} 
over all $0\neq g\in V^*_\Z$ whose coordinates are of size at most $O(q^{1+\epsilon}/H)$, and this in turn can be bounded by $O(H^{n}(q^{1+\epsilon}/H)^{n}M(\Psi_q))=O(q^{n+\epsilon}M(\Psi_q))$. 

In this section, we first show that the set of monic polynomials of degree $n$ over $\F_p$ having given index---or having splitting type containing a given splitting type $\sigma$---are very well distributed in boxes. We accomplish this by demonstrating cancellation in the Fourier transform of the corresponding characteristic functions.

\begin{proposition}\label{ftgen2}
Let $\sigma=(f_1^{e_1}\cdots f_r^{e_r})$ be a splitting type with $\deg(\sigma)=d$, $\ind(\sigma)=k$, and $\len(\sigma)=\ell$.  Let $w_{p,\sigma}:V^1_{\F_p}\to\C$ be defined by
\begin{align*}
w_{p,\sigma}(f):=&\text{ the number of $r$-tuples $(P_1,\ldots,P_r)$, up to the action of the group of}  \\[-.04in] & \text{  permutations of $\{1,\ldots,r\}$ preserving $\sigma$, such that the $P_i$ are distinct } \\[-.04in] &\text{ irreducible monic polynomials with $\deg P_i=f_i$ for each $i$ and $P_1^{e_1}\cdots P_r^{e_r}\mid f$}.
\end{align*}
Then 
\begin{equation*}
\widehat{w_{p,\sigma}}(g)=
\begin{cases}
	{\displaystyle\frac{p^{-k}}{\small\Aut(\sigma)}\!+\!O(p^{-(k+1)})}	& \!\text{if } {g=0};\\[.125in]
	{O(p^{-(k+1)})}		& \!\text{if $g\neq 0$ and $d<n$}\\[.02in]
O(p^{-(k+1/2)}) & \!\text{if $g\neq 0$ and $d=n$;} 
\end{cases}
\end{equation*}
\end{proposition}

\begin{proof}
We have
\begin{eqnarray}\displaystyle
\widehat{w_{p,\sigma}}(g)&=&\displaystyle\frac{1}{p^{n}}\sum_{f\in V^1_{\F_p}} e^{2\pi i[f,g]/p} w_{p,\sigma}(f)
\label{firstexp2}\\
&=&\displaystyle\frac{1}{p^{n}}\sum_{P_1,\ldots,P_r}
\:\sum_{P_1^{e_1}\cdots P_r^{e_r} \mid f}
e^{2\pi i[f,g]/p}. \label{secondexp2}
\end{eqnarray}

Evaluating (\ref{secondexp2}) for $g=0$ gives 
$\widehat{w_{p,\sigma}}(0)=(p^{-k}/\#\Aut(\sigma))+O(p^{-(k+1)})$.  This is
because 1) the number of possibilities for $P_1,\ldots,P_r$ is 
$(1/\#\Aut(\sigma))p^{\ell}+O(p^{{\ell-1}})$, and 2) the set of $f\in V_{\F_p}$ such that $P_1^{e_1}\cdots P_r^{e_r}$ divides $f$ is a linear subspace $U:=U_{P_1^{e_1},\ldots,P_r^{e_r}}\subset V_{\F_p}$ of codimension~$d$, and the set of monic such $f$ is an affine subspace $U^1\subset U$ of codimension 1; hence the number of $f\in U^1$ is  $p^{n-d}.$ We conclude that
$$\widehat{w_{p,\sigma}}(0)=\frac1{p^n}\left(\frac{p^{\ell}}{\#\Aut(\sigma))}+O(p^{\ell-1})\right)p^{n-d} = {\displaystyle\frac{p^{-k}}{\small\Aut(\sigma)}+O(p^{-(k+1)})}.	$$

Next, assume $g\neq 0$. 
We consider each of the $(1/\#\Aut(\sigma))p^{\ell}+O(p^{{\ell-1}})$  terms in the outer sum separately. 
We thus fix $P_1,\ldots,P_r$.
Let $U\subset V_{\F_p}$ denote again the codimension $d$ linear subspace of all $f\in V_{\F_p}$ such that  $P_1^{e_1}\cdots P_r^{e_r}$ divides $f$ and $U^1\subset U$ the codimension~1 affine subspace of those $f\in U$ that are monic. Thus $|U^1|=p^{n-d}$.  

We naturally identify \smash{$V^{1*}_{\F_p}$} with the linear subspace of $V^{*}_{\F_p}$ that kills the monomial $f(x,y)=x^n$.
The set of {$g\in V^{1*}_{\F_p}$} such that $[f,g]$ is a constant (dependent only on $g$) for all $f\in U^1$ is a subspace \smash{$W^1:=W^1_{P_1^{e_1},\ldots,P_r^{e_r}}\subset V^{1*}_{\F_p}$} of dimension~$d$. 
If $g\in W^1$, then the contribution of the inner sum for this fixed choice of $P_1,\ldots,P_r$ is thus \smash{$e^{2\pi i[U^1,g]}p^{n-d}$}. 
Meanwhile, if $g\notin W^1$, 
then, for any $c\in\F_p$, the set $\{f\in U^1: [f,g]=c\}$ is a codimension~1 affine subspace of~$U^1$, and so the contribution of the inner sum is a union of $p^{n-d-1}$ complete sums and hence is~$0$. 

We claim that the number of choices of $P_1,\ldots,P_r$ for which $g\in W^1_{P_1^{e_1},\ldots,P_r^{e_r}}$ is $O(p^{\ell-1})$, if $d<n$. 
Indeed, $g\in W^1_{P_1^{e_1},\ldots,P_r^{e_r}}$ if and only if $g$ annihilates the subspace $U^0_{P_1^{e_1},\ldots,P_r^{e_r}}:=\{f-f': f,f'\in U^1_{P_1^{e_1},\ldots,P_r^{e_r}}\}$ of the subspace $V^0_{\F_p}$ of polynomials in $V_{\F_p}$ of degree $\leq n-1$. The subspace $U^0_{P_1^{e_1},\ldots,P_r^{e_r}}\subset V^0_{\F_p}$ has cardinality $p^{n-d}$, and $U^0_{P_1^{e_1},\ldots,P_r^{e_r}}$ determines $P_1^{e_1}\cdots P_r^{e_r}$.  The subspace $\Ann^0(g):=\{f\in V^0_{\F_p}:[f,g]=0\}$ has $\F_p$-dimension~$n-1$; hence the number of polynomials in the hyperplane $\Ann^0(g)$ having index at least $k$ is $O(p^{n-1-k})$, as when $p>n$ none of the irreducible components of the variety in $V^0$ of codimension $\geq k$ consisting of polynomials having index~$\geq k$ lie in any hyperplane in $V^0$. Since an element of $V^0_{\F_p}$ can lie in at most $O_n(1)$ $U^0_{P_1^{e_1},\ldots,P_r^{e_r}}$'s, it follows that the maximum number of possibilities for $P_1^{e_1},\ldots,P_r^{e_r}$, such that $U^0_{P_1^{e_1},\ldots,P_r^{e_r}}\subset \Ann(g)$, is $O(p^{n-1-k}/p^{n-d})=O(p^{\ell-1})$.

Thus, if $d<n$, then 
$$|\widehat{w_{p,\sigma}}(g)|=\frac1{p^{n}}\cdot p^{n-d}\cdot O(p^{\ell-1}) = O(p^{-(k+1)}).	$$

If $d=n$, then every $U^1:=U^1_{P_1^{e_1},\ldots,P_r^{e_r}}$ contains just one element, namely, $P_1^{e_1}\cdots P_r^{e_r}$, and so every 
$g\in V^{1*}_{\F_p}$ is contained in every $W^1:=W^1_{P_1^{e_1},\ldots,P_r^{e_r}}$.
In this case, we may apply the Weil bound~\cite{Weil} on exponential sums to obtain a nontrivial estimate as follows. We partition all polynomials  $P(x):=P_1^{e_1}(x)\cdots P_r^{e_r}(x)$ into orbits of size $p$ under the action of translation $x\mapsto x+c$ for $c\in\F_p$.  We then consider the elements of each orbit together in the sum (\ref{secondexp2}). 
Given such a polynomial $P(x)$, if $g=(b_1,\ldots,b_n)\neq 0$, then $[P(x+c),g]$ is a nonconstant univariate polynomial $Q$ in $c$ of degree $m$, where $m$ is the maximum index such that $b_m\neq 0$. The contributions of the inner sums in (\ref{secondexp2}) corresponding to $P(x)$ and its translates $P(x+c)$ add up to 
$$\sum_{c\in\F_p} e^{2\pi i [P(x+c),g]} = \sum_{c\in\F_p} e^{2\pi i Q(c)}$$
which is at most $(m-1)p^{1/2}$ in absolute value by the Weil bound.  
Summing over $P(x)=P_1^{e_1}(x)\cdots P_r^{e_r}(x)$ (one from each equivalence class under the action of translation $x\mapsto x+c)$ then yields
$$|\widehat{w_{p,\sigma}}(g)|=O(p^{-n}p^{\ell-1}p^{1/2})=O(p^{-(k+1/2)}),$$
which improves again upon the trivial bound $O(p^{-n}p^{\ell})=O(p^{-k}),$ as desired.
\end{proof}

\vspace{-.1in}

\begin{corollary}\label{index2forms2}
Let $p$ be a prime.  The number of integral monic polynomials of degree $n$
in $[-H,H]^{n}$ 
that, modulo $p$, have index at least $k$ is at most $O(H^{n}/p^k+p^{n-k-1/2+\epsilon})$. 
\end{corollary}

\begin{corollary}\label{Dequi2}
Let $D$ be a positive integer with prime factorization $D=p_1^{k_1}\cdots p_m^{k_m}$ and let $C=p_1\cdots p_m$. 
The number of integral monic polynomials of degree $n$ in $[-H,H]^{n}$ 
that, modulo $p_i$, have index at least $k_i$ for $i=1,\ldots,m$ is at most $O_\epsilon(c^{\omega(C)}H^{n}/D+C^\epsilon\prod_{i=1}^mp_i^{n-k_i-1/2})$. 
\end{corollary}

\begin{corollary}\label{fbound2}
Let $0<\delta < 1/(2n-1)$. Let $D$ be a positive integer with prime factorization $D=p_1^{k_1}\cdots p_m^{k_m}$ such that 
$C=p_1\cdots p_m<H^{1+\delta}$.
For each $i=1,\ldots,m$, let $k_i>1$ be a positive integer. 
Then the number of integral monic polynomials of degree $n$ 
in $[-H,H]^{n}$ that, modulo~$p_i$, have index at least $k_i$, is at most $O(c^{\omega(C)} H^{n}/D).$
\end{corollary}

Corollaries~\ref{index2forms2}--\ref{fbound2} follow from Proposition~\ref{ftgen2} 
just as Corollaries~\ref{index2forms}--\ref{fbound}
followed from Proposition~\ref{ftgen}.

\section{Proof of Theorem~\ref{main1}}\label{main1pf}

In this section, we prove van der Waerden's Conjecture on random Galois groups as stated in Theorem~\ref{main1}, namely, that $E_n(H)=O(H^{n-1})$. 
In order to apply a suitable sieve argument to obtain this bound, we divide the set of irreducible monic integer polynomials $f(x)=x^n+a_1x^{n-1}+\cdots+a_n$ of degree $n$, such that $|a_i|<H$ for all $i$ and $G_f<S_n$ (which, as discussed in the introduction, may be assumed to be primitive), into three subsets.  

To define these three cases for $f$, let $K_f:=\Q[x]/(f(x))$,
let $C$ denote the product of ramified primes in $K_f$, and let $D=|\Disc(K_f)|$. 
Let $\delta>0$ be a small constant, say $\frac1{2n}.$

\subsection*{Case I: $C\leq H^{1+\delta}$ and $D > H^{2+2\delta}$}

We first consider those $f$ for which the product $C$ of ramified primes in $K_f:=\Z[x]/(f(x))$ is at most $H^{1+\delta}$, but the absolute discriminant $D$ of $K:=K_f$ is greater than $H^{2+2\delta}$.  

By Proposition~\ref{valindcon} and the paragraph following it, $D$ must be squarefull as we have assumed that $G_f<S_n$ is primitive.  Given such a $D$, the polynomials $f$ such that $|\Disc(K_f)|=D$ satisfy congruence conditions modulo $C$ that have density $O(\prod_{p\mid C} c/p^{v_p(D)})=O(c^{\omega(D)}/D)$ for a suitable constant $c>0$. 

If $C<H$, then the number of such $f$ within the box $\{|a_i|<H\}$ of sidelength~$H$ can be bounded directly to be $O(c^{\omega(D)}H^n/D)$, because the modulus of the congruence conditions being imposed is smaller than the sidelength of the box. To weaken the condition $C<H$, we may use Fourier analysis, as carried out in Section~\ref{Fourier}.  By Corollary~\ref{fbound2}, we obtain the same estimate $O(c^{\omega(D)}H^{n}/D)$ for the number of monic integer polynomials $f(x)=x^n+a_1x^{n-1}+\cdots+a_n$ of degree $n$ satisfying $|a_i|<H$ for all $i$, $C<H^{1+\delta}$, and $\Disc(K_f)=D$. 

Summing $O(c^{\omega(D)}H^n/D)$ over all squarefull $D>H^{2+2\delta}$ then gives the desired estimate $O(H^{n-1})$ in this case:
$$\sum_{D>H^{2+2\delta} \:\mathrm{ squarefull}} O(c^{\omega(D)}H^{n}/D) = O_\epsilon(H^{n-(1+\delta)+\epsilon}) =O(H^{n-1}).$$

\subsection*{Case II: $C\leq H^{1+\delta}$ and $D \leq  H^{2+2\delta}$}

We next consider those $f$ for which the product $C$ of ramified primes in $K_f:=\Z[x]/(f(x))$ is at most $H^{1+\delta}$ and the absolute discriminant $D$ of $K_f$ is at most $H^{2+2\delta}$.  

In this case, $K=K_f$ is a number field of degree $n$ having absolute discriminant at most $H^{2+2\delta}$.  The number of number fields $K$ of degree $n$ and absolute discriminant at most~$X$, by a result of Schmidt~\cite{Schmidt},\footnote{Improved results for $n$ sufficiently large have been obtained by Ellenberg--Venkatesh~\cite{EV}, Couveignes~\cite{Couveignes}, and most recently, Lemke Oliver and Thorne~\cite{LOT2}.} is $O(X^{(n+2)/4})$. By a result of Lemke Oliver and Thorne~\cite[Proposition~2.2]{LOT}, a number field $K$ of degree $n$ with associated Galois group primitive arises for at most $$O(H\log^{n-1}\!H/|\Disc(K)|^{1/(n(n-1))})$$ such~$f$.  The total number of $f$ in this case is thus at most 
$$O\left(\frac{H^{(2+2\delta)(n+2)/4}\cdot H\log^{n-1}\!H}{H^{(2+2\delta)/(n(n-1))}}\right)=O(H^{n-1})$$ 
for $n\geq 6$.

The exact results known regarding the density of discriminants of quintic fields~\cite{dodpf}---namely, the number of number fields of degree $\leq 5$ and absolute discriminant at most $X$ is $\sim cX$ for an explicit constant $c$---can be used to handle the case $n=5$ as well: the estimate $O((H^{(2+2\delta)(n+2)/4})$ for the number of possible $K$ is replaced by $O(H^{2+2\delta})$, and so the total number of $f$ is then $$O\left(\frac{H^{2+2\delta}\cdot H\log^{n-1}\!H}{H^{(2+2\delta)/(n(n-1))}}\right)=O(H^{n-1})$$ for $n=4,5$. 
 
 \subsection*{Case III: $C> H^{1+\delta}$ $\implies$ $D >  H^{2+2\delta}$}
 
Finally, we consider those $f$ for which the product $C$ of ramified primes in $K_f$ is greater than $H^{1+\delta}$, which implies (by Proposition~\ref{valindcon}) that $D=|\Disc(K_f)|>H^{2+2\delta}$. 
 
 Fix such an $f$. Then for every prime $p \mid C$, the polynomial $f$ has  either at least a  triple root or at least a pair of double roots modulo $p$.  Thus changing $f$ by a multiple of~$p$ does not change the fact that $p^2\mid \Disc(f)$.  That is,  in the language of \cite{geosieve},  $\Disc(f)$ is a multiple of $p^2$ for ``mod $p$ reasons".  For polynomials $h$ that takes a value that is a multiple of $p^2$ for mod $p$ reasons, we have the following proposition:
 
\begin{proposition}\label{geoprop}
 If $h(x_1,\ldots,x_n)$ is an integer polynomial, such that $h(c_1,\ldots,c_n)$ is a multiple of $p^2$, and $h(c_1+pd_1,\ldots,c_n+pd_n)$ is a multiple of $p^2$ for all $(d_1,\ldots,d_n)\in\Z^n$,  then $\frac\partial{\partial x_n} h(c_1,\ldots,c_n)$ is a multiple of $p$.
\end{proposition}

\begin{proof}
Write $h(c_1,\ldots,c_{n-1},x_n)$ as  $$h(c_1,\ldots,c_n) + \textstyle\frac\partial{\partial x_n} h(c_1,\ldots,c_n)(x_n-c_n)+(x_n-c_n)^2r(x),$$
where $r(x)$ is an integer polynomial. 
If we set $x_n$ to be an integer $d_n\equiv c_n$ (mod $p$), then the first and last terms above are always multiples of $p^2$, and hence the middle term must be so as well, implying that 
$\textstyle\frac\partial{\partial x_n} h(c_1,\ldots,c_n)$ must be a multiple of $p$. 
\end{proof}

\noindent
Proposition~\ref{geoprop} applied to $\Disc(f)$ implies that $\frac\partial{\partial a_n} \!\Disc(f)$ is a multiple of $C$;  hence so is $$\Disc_{a_n}(\Disc_x(f(x)))=\Res_{a_n}(\Disc(f),\textstyle\frac\partial{\partial a_n} \!\Disc(f)).$$
For such $f$, the polynomial $\mathrm{DD}(a_1,\ldots,a_{n-1})\!:=\!\Disc_{a_n}(\Disc_x(f(x)))$ is a multiple of~$C$. (Note that $\mathrm{DD}$ is not identically zero, as may be seen by examination of $f(x)=x^n+a_{n-1}x+a_n$, which has discriminant $n^na_n^{n-1}-(n-1)^{n-1}(-a_{n-1})^n$; explicit formulae for such iterated discriminants in general have been given in~\cite{LM}.)

The number of $a_1,\ldots,a_{n-1}\in[-H,H]^{n-1}$ such that $\mathrm{DD}(a_1,\ldots,a_{n-1})=0$
is $O(H^{n-2})$ (by, e.g., \cite[Lemma~3.1]{geosieve}), and so the number of $f$ with such $a_1,\ldots,a_{n-1}$ is $O(H^{n-1})$.

We now fix $a_1,\ldots,a_{n-1}$ such that $\mathrm{DD}(a_1,\ldots,a_{n-1})\neq 0$.  Then 
$\mathrm{DD}(a_1,\ldots,a_{n-1})$ has at most $O_\epsilon(H^{\epsilon})$ factors $C>H$. 
Once $C$ is determined by $a_1,\ldots,a_{n-1}$ (up to $O_\epsilon(H^\epsilon)$ possibilities),  then the number of solutions for $a_n$ (mod $C$) to $\Disc(f)\equiv0$ (mod $C$) is $(\deg_{a_n}(\Disc(f)))^{\omega(C)}=O_\epsilon(H^\epsilon)$ (as the number of possibilites for $a_n$ (mod $p$) such that $\Disc(f)\equiv0$ (mod $p$) for each $p \mid C$ is at most $\deg_{a_n}(\Disc(f))$ by the Fundamental Theorem of Algebra).
 Since $C>H$, the number of possibilities for $a_n\in[-H,H]$ is also at most $O_\epsilon(H^\epsilon)$,  and so the total number  of $f$ in this case is  $O_\epsilon(H^{n-1+\epsilon})$.

\smallskip
This proves the desired result also in Case III, up to a factor of $O_\epsilon(H^{\epsilon})$.  To remove the $\epsilon$, we first observe that there are two sources for the $\epsilon$ in the argument. 
The first is that the number of factors $C$ of $\mathrm{DD}(a_1,\ldots,a_{n-1})$ is~$O(H^\epsilon)$.
The second is that, for each choice of $a_1,\ldots,a_{n-1}$ such that $\mathrm{DD}(a_1,\ldots,a_{n-1})\neq 0$, and a choice of factor $C\mid \mathrm{DD}(a_1,\ldots,a_{n-1})$, we have proven that there are $O((n-1)^{\omega(C)})=O(H^\epsilon)$ choices for $a_n$.

To eliminate these sources of $\epsilon$, we will choose a suitable factor $C'$ of $C$ that is either close to $H$ in size (in which case we can handle it by a method analogous to Case I),  or has all prime divisors greater than $H^{\delta/2}$ (in which case we can handle it by a method analogous to Case III, but with no $\epsilon$ because $C'$ will have at most a bounded number of prime factors!).

We thus break into two subcases.

\vspace{-.05in}
\subsubsection*{Subcase (i): $A=\displaystyle\prod_{{\scriptstyle p\mid C}\atop{\scriptstyle p>H^{\delta/2}}} p \leq H$}

\vspace{-.05in} 
In this subcase, $C$ has a factor $B$ between $H^{1+\delta/2}$ and $H^{1+\delta}$, with $A\mid B\mid C$. Let $B$ be the largest such factor. 
Let $D':=\prod_{p\mid B} p^{v_p(D)}$. Then $D'>H^{2+\delta}$.
We now carry out the argument of Case I, with $B$ in place of $C$, and $D'$ in place of $D$. (Note that $D$ determines $C$ determines $B$ determines $D'$.) 
Summing $O(2^{\omega(D')}H^n/D')$ over all squarefull $D'>H^{2+\delta}$ then gives the desired estimate $O(H^{n-1})$ in this subcase: 
$$\sum_{D'>H^{2+\delta} \:\mathrm{ squarefull}} O(c^{\omega(D')}H^n/D')= O_\epsilon(H^{n-1-\delta/2+\epsilon}) =O(H^{n-1}).$$

\vspace{-.15in}
\subsubsection*{Subcase (ii): $A=\displaystyle\prod_{{\scriptstyle p\mid C}\atop{\scriptstyle p>H^{\delta/2}}} p > H$}

\vspace{-.05in} 
In this subcase, we carry out the original argument of Case III, with $C$ replaced by~$A$.
We have $A\mid \mathrm{DD}(a_1,\ldots,a_{n-1}):=\Disc_{a_n}(\Disc_x(f(x)))$. 

Fix $a_1,\ldots,a_{n-1}$ such that $\mathrm{DD}(a_1,\ldots,a_{n-1})\neq 0$.   Being bounded above by a fixed power of $H$, we see that 
$\mathrm{DD}(a_1,\ldots,a_{n-1})$ can have at most a {\bf bounded} number of possibilities for the factor $A$ (since all prime factors of $A$ are bounded below by a fixed positive power of $H$)!
Once $A$ is determined by $a_1,\ldots,a_{n-1}$,   then the number of solutions for $a_n$ (mod $A$) to $\Disc(f)\equiv0$ (mod $A$) is $O((n-1)^{\omega(A)})=O(1)$. 
 Since $A>H$, the total number  of $f$ in this subcase is also   $O(H^{n-1})$.
 
 \vspace{.1in}
 This completes the proof of Theorem~\ref{main1}.

\section{Proof of Theorem~\ref{main2}}
  
In this section, given a permutation group $G\subset S_n$, we take the methods underlying the three cases in the previous section to their natural limits in order to bound $N_n(G,H)$.  We will work out the maximal regimes for the application of the methods behind each of the three cases, then ultimately choose the method that optimizes bounds whenever the methods overlap. 

To this end, let $G\subset S_n$ be a fixed permutation group.  As in the introduction, let $a=a(G)$ denote a constant such that the number of number fields of degree $n$, associated Galois group $G$, and absolute discriminant at most $X$ is $O(X^{a(G)})$.  
As in Section~\ref{main1pf}, given an integer polynomial $f$ having degree $n$ and nonzero discriminant $\Disc(f)$, let $K_f$ denote the \'etale $\Q$-algebra $\Q[x]/(f(x))$, let $C$ denote the product of ramified primes in $K_f$, and let $D:=\Disc(K_f)$ denote the absolute discriminant of $K_f$. Let $k=\ind(G)$.

The methods behind the cases described below correspond to those in Section~\ref{main1pf}, although the regions in $(C,D)$-space here may be somewhat different than those in the previous section. Let  $0< \delta < 1/(n-1)$ be any constant, and let $Y\ll H^{2n-2}$ be a constant that will be optimized later. 

\subsection*{Case I: $C< H^{1+\delta}$ and $D>Y$}

Given an integer $D>Y$, the polynomials $f$ such that $|\Disc(K_f)|=D$ satisfy congruence conditions modulo $C$ that have density $O(c^{\omega(D)}/D)$, where $c=q(k,n-k)$. 

If $C<H$, then the number of such $f$ within the box $\{|a_i|<H\}$ of sidelength~$H$ can be bounded directly to be $O(c^{\omega(D)}H^n/D)$, because the modulus of the congruence conditions being imposed is smaller than the sidelength of the box. To weaken the condition $C<H$, we again use Fourier analysis, as carried out in Section~\ref{Fourier}.  By Proposition~\ref{ftgen}, we obtain the same estimate $O(c^{\omega(D)}H^n/D)$ for the number of monic integer polynomials $f(x)=x^n+a_1x^{n-1}+\cdots+a_n$ of degree $n$ satisfying $|a_i|<H$ for all $i$, $C<H^{1+\delta}$, and $\Disc(K_f)=D$. 

Summing $O(c^{\omega(D)}H^{n}/D)$ over all $k$-powerful $D$ satisfying $Y<D\ll H^{2n-2}$ then yields the estimate
\begin{equation}\label{caseibest}
\sum_{\substack{{Y<D\ll H^{2n-2}}\\[.02in]{D \:\text{is}\:k\text{-powerful}}}}\!\!\!\!\!\!\!\! O(c^{\omega(D)}H^{n}/D) 
= O(Y^{-(k-1)/k}H^{n}\log^{c-1}\!\!H).
\end{equation}
 
\subsection*{Case II: $D\leq Y$}

If $D\leq Y$, then the number of possibilities for $K_f$ is at most $O(Y^{a(G)})$.  Hence, by the aforementioned result of Lemke Oliver and Thorne~\cite{LOT}, the number of possible $f$ such that $D=|\Disc(K_f)|\leq Y$ is 
\begin{equation}\label{caseiibest}
O\left(Y^{a(G)}\cdot \frac{H\log^{n-1}\!H}{Y^{u}}\right)=O\left(Y^{a(G)-u}H\log^{n-1}\!H\right),
\end{equation}
where $u=1/(n(n-1))$ if $G$ is primitive and $u=0$ otherwise.

\subsection*{Case III: $C>H^{1+\delta}$}

In this case, we show that the number of $f$'s is 
\begin{equation}\label{caseiiibest}
O(H^{n+1-k}).
\end{equation}
As in the proof of Theorem~\ref{main1}, we break into two subcases.

\vspace{-.05in}
\subsubsection*{Subcase (i): $A=\displaystyle\prod_{{\scriptstyle p\mid C}\atop{\scriptstyle p>H^{\delta/2}}} p \leq H$}

\vspace{-.05in} 
In this subcase, $C$ again has a factor $B$ between $H^{1+\delta/2}$ and $H^{1+\delta}$, with $A\mid B\mid C$. We let $B$ be the largest such factor, and let $D':=\prod_{p\mid B} p^{v_p(D)}$. Then $D'>H^{k+k\delta/2}$.
We now carry out the argument of Case I, with $B$ in place of $C$, and $D'$ in place of $D$. 
Summing $O(H^nc^{\omega(D')}/D')$ over all $k$-powerful $D'>H^{k(1+\delta/2)}$ then gives the desired estimate (\ref{caseiiibest})  in this subcase: 
$$
\sum_{D'>H^{k(1+\delta/2)} \;k\!\operatorname{th-power-full}}
O(c^{\omega(D')}H^n/D')= O_\epsilon(H^{n-(k-1)(1+\delta/2)+\epsilon}) =O(H^{n+1-k}).$$

\vspace{-.15in}
\subsubsection*{Subcase (ii): $A=\displaystyle\prod_{{\scriptstyle p\mid C}\atop{\scriptstyle p>H^{\delta/2}}} p > H$}

\vspace{-.05in} 
In this subcase, we observe that, by Remark~\ref{irremark}, the variety $V_k$ of monic polynomials of degree $n$ having index $k$ has dimension $n-k$ in the $\A_n$ of all monic polynomials $f(x)=x^n+a_1x^{n-1}+\cdots+a_n$ of degree $n$. Therefore, $V_k$ is contained in the zero set of $k$ polynomials $g_1,g_2,\ldots,g_k$ on $\A^n$, and these $k$ polynomials may be taken to be defined over $\Z$ (indeed, by the same argument as in Proposition~\ref{geoprop}, we may take these $k$ polynomials to be $\Disc(f)$ and its first $k-1$ derivatives with respect to $a_n$).  By elimination theory / taking successive resultants, we may assume that $g_i$ only involves the variables $a_1,\ldots,a_{n-k+i}$. 

Let 
$$W_i(H):=\{f(x)=x^n+a_1x^{n-1}+\cdots+a_n: C>H^{1+\delta}, \,A>H, \,g_1=\cdots=g_i=0, \,g_{i+1}\neq 0\}$$
where, for convenience, we define $g_{k+1}=1$.  We show that $|W_i(H)|=O(H^{n-k+1})$ for each $i=0,\ldots,k$.  \pagebreak 

Thus fix such an $i$.
The number of values of $a_1,\ldots,a_{n-k+i}\in[-H,H]^{n-k+i}$ such~that $g_1=\cdots=g_i=0$ is $O(H^{n-k})$ (by, e.g., \cite[Lemma~3.1]{geosieve}). This immediately yields the desired estimate $|W_i(H)|=O(H^{n-k})=O(H^{n-k+1})$ when $i=k$.  If $i<k$, then there are $O(H)$~choices for $a_{n-k+i+1}$ such that $g_{i+1}\neq 0$.  This gives the desired estimate $|W_i(H)|=O(H^{n-k}\cdot H)=O(H^{n-k+1})$ when $i=k-1$.  If $i<k-1$, then, since $g_{i+1}(a_1,\ldots,a_{n-k+i+1})$ is nonzero and is at most a bounded power of $H$, we see that 
$g_{i+1}(a_1,\ldots,a_{n-k+i+1})$ determines at most a {\bf bounded} number of possibilities for the factor $A$. Once $A$ is determined by $a_1,\ldots,a_{n-k+i+1}$,  the number of possibilities for $a_{n-k+i+2},\ldots,a_n$ (mod~$A$) is also bounded, by Proposition~\ref{indexcodim}!
Hence we conclude, for general $i\in\{0,\ldots,k\}$, that $|W_i(H)|=O(H^{n-k}\cdot H)=O(H^{n-k+1})$, yielding the desired estimate (\ref{caseiiibest}) in this subcase as~well.
  
\subsection*{Conclusion}

If we set $Y\asymp H^{2n-2}$, then Case II alone covers all possibilities, and yields the bound (\ref{caseiibest}) with $Y=H^{2n-2}$. Otherwise, we must sum the estimates from Cases I--III.  The bound (\ref{caseiiibest}) is fixed at $O(H^{n+1-k})$, and will be less than the aforementioned bound coming from Case II if
\begin{equation}\label{abound}
a(G)>\frac{n-k+2/n}{2n-2}.
\end{equation}
If (\ref{abound}) holds, then the sum of the bounds of Cases I and II is minimized for the value of $Y$ where (\ref{caseibest}) and (\ref{caseiibest}) are equal, namely, 
\begin{equation}\label{meq}
Y=H^{\textstyle{\frac{n-1}{a(G)+1-1/k-u}}}(\log H)^{\textstyle{\frac{c-n}{a(G)+1-1/k-u}}}.
\end{equation}
This then yields the bound
\begin{equation}\label{bothiandiibound}
H^{n-(n-1)\textstyle{\frac{1-1/k}{a(G)+1-1/k-u}}}\log^b\! H
\end{equation}
in both Cases I and II, where $b=(c-1)-(c-n){\textstyle{\frac{1-1/k}{a(G)+1-1/k-u}}}<c$. We conclude that
\begin{equation*}\label{nngbound}
N_n(G,H)=O\Bigl(\min\Bigl\{H^{n+1-k}+H^{n- {(n-1){\textstyle{\frac{1-1/k}{a(G)+1-1/k-u}}}}}\log^b \!H,H^{(2n-2)(a(G)-u)+1}\log^{n-1}\!H\Bigr\}\Bigr),
\end{equation*}
yielding Theorem \ref{main2}. 

\section{Proof of Corollary~\ref{main3}}\label{proofofmain3}

\noindent 
{\bf (a):} To prove Corollary~\ref{main3}(a), we note that by Theorem~\ref{genindbound} (or using the explicit classifications of groups having elements moving few letters as in the works of Manning~\cite{Manning2p,Manning3p,Manning10,Manning68,Manning12,Manning14}, which do not use CFSG), we have for any primitive permutation group $G$ on $n$ letters that is not equal to $A_n$ or $S_n$ that $\ind(G)\geq 3$ if $n\geq 9$, \:$\ind(G)\geq 4$ if $n\geq 16$, and $\ind(G)\geq 5$ if $n\geq 25$. 
By Theorem~\ref{primcount}, we may take $a(G)=\frac{n+2}4-1+\frac{1}{\ind(G)}+\epsilon$ for any $\epsilon>0$. 
Plugging these bounds on $\ind(G)$ and $a(G)$ into Theorem~\ref{main2} then yields 
$$
N_n(G,H)\ll H^{n+1-3}+{H}^{n-(n-1)\textstyle\frac{1-1/3}{(n+2)/4-u+\epsilon}}\log^{c}\!H \ll H^{n-2}$$
if $n\geq 10$ (with $N_9(G,H)\ll H^{7.06}$), 
$$
N_n(G,H)\ll H^{n+1-4}+{H}^{n-(n-1)\textstyle\frac{1-1/4}{(n+2)/4-u+\epsilon}}\log^{c}\!H \ll H^{n-2.5}
$$
if $n\geq 16$, and 
$$
N_n(G,H)\ll H^{n+1-5}+{H}^{n-(n-1)\textstyle\frac{1-1/5}{(n+2)/4-u+\epsilon}}\log^{c}\!H \ll H^{n-3}
$$
if $n\geq 28$.

The final assertion of Corollary~\ref{main3}(a) illustrates the power of the results of Liebeck and Saxl~\cite{LS} and Guralnick and Magaard~\cite{gm}, using CFSG, when gets $n$ larger, as well as the advantage of separating the cases of elemental and non-elemental groups.  
First, applying the already impressive result of Babai~\cite{Babai}, we see by Remark~\ref{cfsg} that $\ind(G)\geq 9$ for a primitive group $G\subset S_n$ not equal to $A_n$ or $S_n$ if $n\geq 1090$. We therefore have 
$$
N_n(G,H)\ll H^{n+1-9}+{H}^{n-(n-1)\textstyle\frac{1-1/9}{(n+2)/4-u+\epsilon}}\log^{c}\!H \ll H^{n-3.5} 
$$
if $n\geq 1090$. 

To reduce the threshold 1090 to 53, we note that by Remark~\ref{basicindboundgen} and Theorem~\ref{genindbound}, we have $\ind(G)\geq 13$ for any elemental $G\subset S_n$ not equal to $A_n$ or $S_n$, and $\ind(G)\geq 7$ for any non-elemental~$G$ when $n\geq 53$. By Theorem~\ref{primcount} and (\ref{nonecount2}), we may take $a(G)=\frac{n+2}4-1+\frac{1}{\ind(G)}+\epsilon$ if $G$ is elemental and $a(G)=4$ if $G$ is non-elemental. 
Hence, if $G$ is a primitive permutation group on $n\geq 53$ letters, then
$$
N_n(G,H)\ll H^{n+1-13}+{H}^{n-(n-1)\textstyle\frac{1-1/13}{(n+2)/4-u+\epsilon}}\log^{c}\!H \ll H^{n-3.5}
$$
when $G$ is elemental and 
\begin{equation}\label{none2}
N_n(G,H)\ll H^{n+1-7}+{H}^{n-(n-1)\textstyle\frac{1-1/7}{4-u}}\log^{c}\!H \ll H^{n-6}
\end{equation}
when $G$ is non-elemental.

\bigskip
\noindent 
{\bf (b):} If $G\subset S_n$ is non-elemental, then by Theorem~\ref{nonecount}, 
we may take $a(G)=O(\log^2\!n/\!\sqrt{n})$ where the implied constant is absolute. Therefore, 
$$N_n(G,H)\ll H^{(2n-2)O(\log^2\!n/\!\sqrt{n})+1}\log^{c}\!H \ll H^{b\sqrt{n}\log^2\!n}$$
for some absolute constant $b>0$. 

\bigskip
\noindent 
{\bf (c):} If $G\subset S_n$ is elemental, then by Theorem~\ref{basicindbound}, we have $\ind(G)\geq 3n/14$, while by Theorem~\ref{slot}, we may take $a(G)=r\log^2\!n$ for some constant $r>0$. We therefore have
$$
N_n(G,H)\ll H^{n+1-3n/14}
+{H}^{n-(n-1)\textstyle\frac{1-14/(3n)}{c\log^2\!n}}
\log^{c}\!H \ll H^{n-bn/\!\log^2n}$$
for some absolute constant $b>0$. 

If we apply only the bound $\ind(G)\geq \frac14(\sqrt{n}-1)$ as in Remark~\ref{cfsg}, which does not use the classification of finite simple groups, then the above argument would yield 
$$
N_n(G,H)\ll H^{n+1-\frac14(\sqrt{n}-1)}
+{H}^{n-(n-1)\textstyle\frac{1-4/(\sqrt{n}-1)}{c\log^2\!n}}
\log^{c}\!H \ll H^{n-b\sqrt{n}}$$
for some absolute constant $b>0$. 


\section{Proof of Theorem~\ref{intransitivethm}}

For a monic integer polynomial $f(x)=x^n+a_{1}x^{n-1}+\cdots+a_n$ of degree $n$, define the height $H(f)$ of $f$ by $\max\{|a_i|\}_{1\leq i\leq n}$, and define the Mahler measure $M(f)$ by $\prod_{k=1}^n \max\{1,|r_k|\}$, where $r_1,\ldots,r_n$ denote the (not necessarily distinct) roots of $f$ over $\C$. Then it is known that
$$ {n\choose \lfloor n/2\rfloor}^{-1} H(f) \leq M(f)\leq \sqrt{n+1}\: H(f),$$
so $H(f)=\Theta(M(f))$, i.e., $H(f)=O(M(f))$ and $M(f)=O(H(f))$. The function $M(f)$ is obviously multiplicative, i.e., $M(fg)=M(f)M(g)$, and hence so is $H(f)$ up to bounded factors, i.e., $H(fg)=\Theta(H(f)H(g))$. 

Now if $f$ is a polynomial counted by $N_n(G,H)$, where $G$ is a subdirect product of $G_1, \ldots, G_k$, then $f=f_1\cdots f_k$ where the Galois group of the monic integer polynomial $f_i$ is~$G_i$. 
Suppose $N_{n_i}(G_i,H)=O(H^{\alpha_i}\log^{\beta_i}\!H)$, and assume without loss of generality that $\alpha_1\leq \alpha_2\leq \cdots \leq \alpha_k$.  

We will prove the assertion of the theorem by induction on $k$.  Thus assume the result is true for subdirect products of $G_2\times\cdots\times G_{k}$. Let $M_{n_1}(G_1,H_1)$ denote the number of monic polynomials of degree $n$ having height {\it exactly} $H_1$. Then 
\begin{align*}
N_n(G,H)
&=\sum_{H_1\leq H}M_{n_1}(G_1,H_1) \cdot N_{n_2+\cdots+n_{k}}(G_2\times \cdots \times G_{k},H/H_1)
\\
&\ll \sum_{H_1\leq H} M_{n_1}(G_1,H_1) \cdot 
(H/H_1)^{\alpha_{k}}\log^{\sum_{{{}\atop{{\alpha_i=\alpha_{k}}\atop{i>1}}}}(1+\beta_i)-1}\!(H/H_1)\\
&\ll \sum_{H_1\leq H} H_1^{\alpha_1-1}\log^{\beta_1}\!H_1\cdot (H/H_1)^{\alpha_{k}}\log^{\sum_{{{}\atop{{\alpha_i=\alpha_{k}}\atop{i>1}}}}(1+\beta_i)-1}\!(H/H_1)
\\
&< \sum_{H_1\leq H} H_1^{\alpha_1-1-\alpha_{k}}\log^{\beta_1}\!H_1\cdot H^{\alpha_{k}}\log^{\sum_{{{}\atop{{\alpha_i=\alpha_{k}}\atop{i>1}}}}(1+\beta_i)-1}\!H 
\\
&\ll \log^{\kappa(1+\beta_1)}\!H\cdot H^{\alpha_k}\log^{\sum_{{{}\atop{{\alpha_i=\alpha_{k}}\atop{i>1}}}}(1+\beta_i)-1}\!H  
\end{align*}
where $\kappa=1$ if $\alpha_{1}=\alpha_k$ and $0$ otherwise. This yields the desired result.

\subsection*{Acknowledgments}

I am extremely grateful to Andrew Granville, Benedict Gross, Robert Guralnick, Eilidh McKemmie, Robert Lemke Oliver, Danny Neftin, Andrew O'Desky, Peter Sarnak, Jean-Pierre Serre, Arul~Shankar, Artane Siad, Frank Thorne, Xiaoheng Wang, and the anonymous referees for numerous helpful conversations and comments on earlier drafts of this manuscript. The author was supported by a Simons Investigator Grant and NSF grant~DMS-1001828.

\vspace{.25in}
\noindent {Princeton University}

\noindent {e-mail: bhargava@math.princeton.edu}


\begin{thebibliography}{32}

\bibitem{Alberts}
B.\ Alberts, The weak form of Malle's Conjecture and solvable groups (2018), {\tt arXiv:
1804.11318v1}. 

\bibitem{aimgroup}
T.\ C.\ Anderson, A.\ Gafni, R.\ J.\ Lemke~Oliver, D.\ Lowry-Duda, G.\ Shakan, and R.~Zhang, Quantitative Hilbert irreducibility and almost prime values of polynomial discriminants (2021), {\tt arXiv:2107.02914v1}.

\bibitem{Babai}
L.\ Babai, On the order of uniprimitive permutation groups, {\it Ann.\ of Math.} {\bf 113} (1981), 553--568.

\bibitem{dodqf}
M.\ Bhargava, The density of discriminants of quartic rings and
fields, {\it Ann.~of Math.} {\bf 162} (2005), 1031--1063.

\bibitem{dodpf}
M.\ Bhargava, The density of discriminants of quintic rings and
fields, {\it Ann.~of Math.} {\bf 172} (2010), 1559--1591.

\bibitem{mass}
M.\ Bhargava, Mass formulae for extensions of local fields, and conjectures on the density of number field discriminants,
{\it Int.\ Math.\ Res.\ Not.} {\bf 2007}, Article ID rnm052.

\bibitem{geosieve}
M.\ Bhargava, 
{{The geometric sieve and the density of squarefree values of invariant polynomials (2014),}}
 {\tt{arXiv:1402.0031v1}}. 

\bibitem{BW}
M.\ Bhargava and M.M.\ Wood, The density of discriminants of $S_3$-sextic number fields, {\it Proc.\ 
Amer.\ Math.\ Soc.} {\bf 136} (2008), no.\ 5, 1581--1587.

\bibitem{Cohen}
S.\ D.\ Cohen, The distribution of Galois groups and Hilbert's irreducibility theorem, 
{\it Proc.\ Lond.\ Math.\ Soc.} (3) {\bf 43} (1981), 227--250.

\bibitem{Chela}
R.\ Chela, Reducible polynomials, {\it J.\ Lond.\ Math.\ Soc.} {\bf 38} (1963), 183--188.

\bibitem{CD}
S.\ Chow and R.\ Dietmann, Enumerative Galois theory for cubics and quartics,
{\it Adv.\ Math.} {\bf 372} (2020): 107282. 

\bibitem{CD2}
S.\ Chow and R.\ Dietmann, 
Towards van der Waerden's conjecture (2021), 
{\tt arXiv:2106.14593v1}.

\bibitem{Couveignes}
J.-M.~Couveignes, Enumerating number fields (2019), {\tt arXiv:1907.13617v1}.

\bibitem{Dietmann}
R.\ Dietmann, 
On the distribution of Galois groups, 
{\it Mathematika} {\bf 58} (2012), No.~1, 35--44.

\bibitem{Dietmann2}
R.\ Dietmann, 
Probabilistic Galois theory, {\it Bull.\ London Math.\ Soc.} {\bf 45} (2012), No.~3, 453--462.

\bibitem{dm}
 J.\ D.\ Dixon and B.\ Mortimer, {\it Permutation Groups}, Springer-Verlag, New York--Heidelberg--Berlin, 1996.

\bibitem{Dummit}
E.\ Dummit, Counting $G$-extensions by discriminant, {\it Math.\ Res.\ Let.} {\bf 25} (2018), no.~4, 
1151--1172. 

\bibitem{EV}
J.~S.~Ellenberg and A.~Venkatesh, The number of extensions of a number field with
fixed degree and bounded discriminant, {\it Ann.~of Math.~(2)} {\bf 163} (2006), No.~(2), 723--741.

\bibitem{F}
\'E.\ Fouvry and P.\ Koymans, Malle's conjecture for nonic Heisenberg extensions, {\tt arXiv: 2102.09465v1}.

\bibitem{Gallagher}
P.~X.~Gallagher, The large sieve and probabilistic Galois theory, {\it Proceedings of Symposia
in Pure Mathematics XXIII}, A.M.S., 1973, pp.~91--101.

\bibitem{gm}
R.\ Guralnick and K.\ Magaard,  On the minimal degree of a primitive
permutation group, {\it J.~Algebra} {\bf 207} (1998), 127--145. 
  
\bibitem{KM}
J.\ Kl\"uners and G.~Malle, Counting nilpotent Galois extensions, {\it J.\ Reine Angew.\ Math.} {\bf 572} (2004), 1--26.

\bibitem{Knobloch}
H.-W.~Knobloch, Zum Hilbertschen Irreduzibilit\"atssatz, {\it Abh.~Math.~Sem.~Univ.\ Hamburg} {\bf 19} (1955), 176--190.

\bibitem{Knobloch2}
H.-W.~Knobloch, Die Seltenheit der reduziblen Polynome, {\it Jber.~Deutsch.~Math.\ Verein.} {\bf 59}, (1956), 12--19.

\bibitem{LR}
E.\ Larson and L.~Rolen, Upper bounds for the number of number fields with alternating Galois group, {\it Proc.\ Amer.\ Math.\ Soc.} {\bf 141} (2013), no.~2, 499--503.

\bibitem{LM}
D.\ Lazard and S.\ McCallum, Iterated discriminants,
{\it J.\ Symb.\ Comp.} {\bf 44} (2009), no.\ 9, 1176--1193.

\bibitem{LOT}
R.\ J.\ Lemke Oliver and F.\ Thorne, Upper bounds on polynomials with small Galois group, {\it Mathematika} {\bf 66} (2020), No.~4, 1054--1059.

\bibitem{LOT2}
R.~J. Lemke Oliver and F.\ Thorne, Upper bounds on number fields of given degree and bounded discriminant (2020), {\tt arXiv:2005.14110v1}.

\bibitem{LS}
M.\ Liebeck and J.\ Saxl, Minimal degrees of primitive permutation groups, with an
application to monodromy groups of covers of Riemann surfaces, {\it Proc.\ London
Math.\ Soc.} (3) {\bf 63} (1991), 266--314.

\bibitem{Maki}
S.\ M\"aki, On the density of abelian number fields, {\it Ann.\ Acad.\ Sci.\ Fenn.\ Ser.\ A I Math.\ Dissertationes} {\bf 54} (1985).

\bibitem{Malle}
G.\ Malle, On the Distribution of Galois Groups, {\it J.\ Number Theory} {\bf 92} (2002), 315--329. 

\bibitem{Manning2p}
W.\ A.\ Manning, The primitive groups of class $2p$ which contain a substitution of order~$p$ and degree $2p$, {\it Trans.\ Amer.\ Math.\ Soc.} {\bf 4} (1903), 351--357.

\bibitem{Manning3p}
W.\ A.\ Manning, On the primitive groups of class $3p$, {\it Trans.\ Amer.\ Math.\ Soc.} {\bf 6} (1905), 42--47.

\bibitem{Manning10}
W.\ A.\ Manning, On the primitive groups of class ten, 
{\it Amer. J.\ Math.} {\bf 28} (1906), 226--236. 

\bibitem{Manning68}
W.\ A.\ Manning, On the primitive groups of classes six and eight,
{\it Amer. J.\ Math.} {\bf 32} (1910), 235--256. 

\bibitem{Manning12}
W.\ A.\ Manning, On the primitive groups of class twelve, 
{\it Amer. J.\ Math.} {\bf 35} (1913), 229--260. 


\bibitem{Manning14}
W.\ A.\ Manning, The primitive groups of class fourteen, 
{\it Amer. J.\ Math.} {\bf 51} (1929), 619--652. 

\bibitem{Miller}
G.\ A.\ Miller, On the primitive groups of class four, {\it Amer.\ Math. Monthly} {\bf 9} (1902), 63--66.

\bibitem{Schmidt}
W.\ M.\ Schmidt, Number fields of given degree and bounded discriminant, {\it Ast\'erisque}~{\bf 228} (1995), No.~4, 189--195.

\bibitem{Serre}
J-P.\ Serre, {\it Lectures on the Mordell--Weil Theorem}, Aspects of Mathematics {\bf 15}, Springer Fachmedien Wiesbaden, 1997.

\bibitem{SS}
E.\ Stein and R.\ Shakarchi,
{\it Fourier Analysis: An Introduction}, Princeton University Press, 2003. 

\bibitem{TT}
T.\ Taniguchi and F.\ Thorne, 
Levels of distribution for sieve problems in prehomogeneous vector spaces, {\it Math.\ Ann.} 
{\bf 376} (2020), 1537--1559.

\bibitem{vdw}
B.~L.~van der Waerden, Die Seltenheit der reduziblen Gleichungen und der Gleichungen mit Affekt, {\it Monatsh.~Math.\ 
Phys.}, {\bf 43} (1936), No.\ 1, 133--147.

\bibitem{W}
J.\ Wang, Malle's conjecture for $S_n\times A$ for $n=3,4,5$,
{\it Compositio Math.} {\bf 157} (2021), 83--121.

\bibitem{Weil}
A.\ Weil, On some exponential sums, {\it Proc.\ Nat.\ Acad.\ Sci.\ U.S.A.} {\bf 34} (1948),
204--207.

\bibitem{Widmer}
M.\ Widmer, On number fields with nontrivial subfields, 
{\it International Journal of Number Theory} {\bf 7} (2011), No.~3, 695--720. 

\bibitem{Williamson}
A.\ Williamson,  On primitive permutation groups containing a cycle, 
{\it Math.\ Z.} {\bf 130} (1973), 159--162.

\bibitem{Wright}
D.\ J.\ Wright, Distribution of discriminants of abelian extensions, {\it Proc.\ Lond.\ Math.\ Soc.} 
{\bf 58} (1989), 17--50. 

\bibitem{Xiao}
S. Y.\ Xiao, On monic abelian cubics (2019), {\tt arXiv:1906.08625}. 

\bibitem{Zywina}
D.\ Zywina, Hilbert's irreducibility theorem and the larger sieve (2010),
{\tt arXiv.org:1011.6465v1}.

\end{thebibliography}
\end{document}